\documentclass[a4paper,11pt]{amsart}
\usepackage{graphicx}
\usepackage{latexsym, color}
\usepackage{amsthm}
\usepackage{autograph,eepic,epic,latexsym,bezier,amsbsy,color,enumerate,amsfonts,amsmath,amscd,amssymb}
\usepackage[latin1]{inputenc}

\setloopdiam{8}\setprofcurve{7}
\newtheorem{defi}{Definition}[section]
\newtheorem{teo}[defi]{Theorem}
\newtheorem{lem}[defi]{Lemma}
\newtheorem{prop}[defi]{Proposition}

\newtheorem{os}[defi]{Remark}

\def\subjclass#1{{\renewcommand{\thefootnote}{}%
\footnote{\emph{Mathematics Subject Classification (2010):} #1}}}

\begin{document}

\title{The Tutte Polynomial of the Sierpi\'{n}ski and Hanoi graphs}



\maketitle

\author{\begin{center}{ALFREDO DONNO\\ Dipartimento di Matematica,\quad Sapienza Universit\`a di
Roma,\quad Piazzale A. Moro, 2\\ 00185 Roma, Italia. \qquad {\it
E-mail address:}
\texttt{alfredo.donno@sbai.uniroma1.it}}\end{center}}

\author{\begin{center}{DONATELLA IACONO\\ Dipartimento di Matematica,\quad Universit\`a degli
Studi di Bari,\quad Via E. Orabona, 4\\ 70125 Bari, Italia. \qquad
{\it E-mail address:} \texttt{iacono@dm.uniba.it}}\end{center}

\begin{abstract}
We study the Tutte polynomial of two infinite families of finite
graphs: the Sierpi\'{n}ski graphs, which are finite approximations
of the well-known Sierpi\'{n}ski gasket, and the Schreier graphs
of the Hanoi Towers group $H^{(3)}$ acting on the rooted ternary
tree. For both of them, we recursively describe the Tutte
polynomial and we compute several special evaluations of it,
giving interesting results about the combinatorial structure of
these graphs.
\end{abstract}

\section{Introduction}

The Tutte polynomial is a two-variable polynomial which can be
associated with a graph, a matrix, or, more generally, with a
matroid. It has many interesting applications in several areas of
sciences as, for instance, Combinatorics, Probability, Statistical
Mechanics, Computer Science and Biology. It was introduced by W.T.
Tutte \cite{tutte47, tutte54, tutte67} and we will mainly refer to
\cite{bollobas, equivalence, machi, tutte04} as expository papers.

Given a finite graph $G$, its Tutte polynomial $T(G;x,y)$
satisfies a fundamental universal property with respect to the
deletion-contraction reduction of the graph. Hence, any
multiplicative graph invariant with respect to a
deletion-contraction reduction turns out to be an evaluation of
it. This polynomial is quite interesting since several
combinatorial, enumerative and algebraic properties of the graph
can be investigated by considering special evaluations of it. For
instance, one gets information about the number of spanning trees,
spanning connected subgraphs, spanning forests and acyclic
orientations of the graph. Moreover, the Tutte polynomial also
allows to recover the reliability and the chromatic polynomials.
It has also many interesting connections with statistical
mechanical models as the Potts model \cite{potts}, the percolation
\cite{percolation}, the Abelian Sandpile Model \cite{cori,
merino}, as well as with the theory of error correcting codes
\cite{potts}.

In this paper, we study the Tutte polynomial of two infinite
families of finite graphs very close each other: the
Sierpi\'{n}ski graphs $\{\Gamma_n\}_{n\geq 1}$, approximating the
famous Sierpi\'{n}ski gasket, and the Schreier graphs
$\{\Sigma_n\}_{n\geq 1}$ of the Hanoi Towers group $H^{(3)}$,
whose action on the ternary tree models the well-known Hanoi
Towers game on three pegs, and which is an example of self-similar
group. In the last decades, the study of automorphism groups of
rooted trees has been largely investigated: R. Grigorchuk and a
number of coauthors have developed a new exciting direction of
research focusing on finitely generated groups acting by
automorphisms on rooted trees \cite{fractal}. They proved that
these groups have deep connections with the theory of profinite
groups and with complex dynamics. In particular, many groups of
this type satisfy a property of self-similarity, reflected on
fractalness of some limit objects associated with them
\cite{volo}. In \cite{csdi}, the Tutte polynomial is studied for
two examples of Schreier graphs associated with the action of two
self-similar groups on the rooted binary tree, namely the
Grigorchuk group and the Basilica group.

In this paper, we follow a combinatorial approach and we use the
self-similar structure of the graphs (in the sense of
\cite{wagner2}) to recursively investigate the Tutte polynomial of
the families $\{\Gamma_n\}_{n\geq 1}$ and $\{\Sigma_n\}_{n\geq
1}$. It is worth mentioning \cite{biggsetal, noy}, where the
authors consider the Tutte polynomial for recursive families of graphs.\\
\indent We study the Tutte polynomial analyzing all the spanning
subgraphs of these sequences of graphs. Then, we give a partition
of the set of the spanning subgraphs, that allows us to split the
polynomial as a sum of three terms (the same strategy turns out to
be a powerful tool to investigate many combinatorial and
statistical models on them: see, for instance, \cite{chang,
taiwan3, taiwan2, taiwan, noispanning, noiising, noi1, wagner1,
wagner2}). Finally, using self-similarity, we are able to give a
recursive formula for each of these terms (see Theorems
\ref{ricorsivesierpinski} and \ref{noname}). Once we have these
formulas, we are able to show the following properties:
\begin{itemize}
\item \emph{recursive formulas for the Tutte polynomial} (Theorems \ref{ricorsivesierpinski} and
\ref{noname});\smallskip
\item  \emph{description of the reliability polynomial} (Propositions \ref{donatellasierp} and
\ref{donatellahanoi});\smallskip
\item  \emph{computation of complexity} (Propositions
\ref{propcomplexsierp} and \ref{propcomplexhanoi});\smallskip
\item  \emph{number of connected spanning subgraphs} (Propositions \ref{propconnsubgrsierp} and
\ref{propconnsubgrhanoi});\smallskip
\item \emph{number of spanning forests} (Propositions
\ref{spannforsierp} and \ref{spannforesthanoi});\smallskip
\item  \emph{number of acyclic orientations} (Propositions \ref{propacycsierp} and
\ref{propacyclhanoi});\smallskip
\item  \emph{description of the chromatic polynomial} (Propositions
\ref{propchromsierp} and \ref{propchromhanoi});\smallskip
\item  \emph{computation of the partition function of the Ising model} (Theorems \ref{thmisingsierp} and \ref{thmisinghanoistrano}).
\end{itemize}

Some of these combinatorial properties were already known in
literature, and so our main result, which consists in the
description of the Tutte polynomial, collects these aspects in a
more general and stronger context.\\
\indent  The paper is organized as follows. Section \ref{Section
preliminare} contains some preliminary material on the Tutte
polynomial and on the theory of automorphism groups of rooted
regular trees. In Section \ref{sezione serpin}, we study the Tutte
polynomial of the Sierpi\'{n}ski graphs $\{\Gamma_n\}_{n\geq 1}$
and analyze many applications. In particular, we recover that, for
every $n\geq 1$, the graph $\Gamma_n$ is uniquely $3$-colorable
(Proposition \ref{colorabilitysierp}). Section \ref{sezione hanoi}
is devoted to the analysis of the Schreier graphs
$\{\Sigma_n\}_{n\geq 1}$ of the Hanoi Towers group $H^{(3)}$. We
also underline the very close structure of the graphs
$\{\Gamma_n\}_{n\geq 1}$ and $\{\Sigma_n\}_{n\geq 1}$ by pointing
out the relationship between $T(\Gamma_n;x,y)$ and
$T(\Sigma_n;x,y)$ (Proposition \ref{sierp-hanoi}).

\smallskip

While submitting this paper, we became aware of a  recent
preprint \cite{alvarez}, submitted to arXiv after our one, where
similar computations are performed.

\section{Preliminaries}\label{Section preliminare}

\subsection{The Tutte polynomial}\label{2.1}

Throughout the paper, we deal with graphs which are connected and
finite. Moreover, both multiple edges and multiple loops are
allowed. As usual, $G=(V(G),E(G))$ denotes a graph with vertex set
$V(G)$ and edge set $E(G)$; we will often write $V$ and $E$, when
there is no risk of confusion, and so $G=(V,E)$. Moreover, we
denote by $E_n$ the graph with $n$ vertices and no edges, and by
$K_n$ the complete graph on $n$ vertices. A subgraph
$A=(V(A),E(A))$ of a graph $G=(V(G),E(G))$ is said \emph{spanning}
if the condition $V(A)=V(G)$ is satisfied. In particular, a
\emph{spanning subtree} of $G$ is a spanning subgraph of $G$ which
is a tree, a \emph{spanning forest} of $G$ is a spanning subgraph
of $G$ which is a forest. The number of spanning trees of a graph
$G$ is called \emph{complexity} of $G$ and is denoted by
$\tau(G)$.  It is interesting to study complexity when the system
grows. More precisely, given a sequence $\{G_n\}_{n\geq 1}$ of
finite graphs with complexity $\tau(G_n)$, such that $|V(G_n)|\to
\infty$, the limit
$$
\lim_{|V(G_n)| \to \infty}\frac{\log \tau(G_n)}{|V(G_n)|},
$$
when it exists, is called the \emph{asymptotic growth
constant} of the spanning trees of $\{G_n\}_{n\geq 1}$ \cite{lyons}.\\
Finally, let $k(G)$ be the number of connected components of $G$.

\begin{defi}
Let $A$ be a spanning subgraph of $G$, then the rank $r(A)$ and
the nullity $n(A)$ of $A$ are defined as
\[
r(A)=|V(A)| -k(A)= |V(G)|-k(A) \quad \mbox{ and } \quad
n(A)=|E(A)| - r(A)=|E(A)|-|V(A)| + k(A).
\]
\end{defi}

\begin{defi}[Spanning subgraphs]\label{defspanning}
Let $G=(V,E)$ be a graph. The Tutte polynomial $T(G;x,y)$ of $G$
is defined as
\begin{eqnarray}\label{defsubgraphs}
T(G;x,y)= \sum_{A\subseteq G}(x-1)^{r(G)-r(A)}(y-1)^{n(A)},
\end{eqnarray}
where the sum runs over all the spanning subgraphs $A$ of $G$.
\end{defi}

The Tutte polynomial can be also defined by a recursion process
given by deleting and contracting edges. We recall that, given
$G=(V,E)$, the graph $G\setminus e=(V,E-\{e\})$ is obtained from
$G$ by deleting the edge $e\in E$. The graph obtained by
contracting an edge $e\in E$ is the result of the identification
of the endpoints of $e$ followed by removing $e$. We denote it by
$G/ e$. Finally, we recall that an edge in a connected graph is a
\emph{bridge} if its deletion disconnects the graph, it is a
\emph{loop} if its endpoints coincide.

\begin{defi}[Deletion-Contraction]\label{contracting}
Let $G=(V,E)$ be a graph. The Tutte polynomial $T(G;x,y)$ of $G$
is defined as
$$T(G;x,y)=
\left\{
\begin{array}{ll}
   1  & \hbox{if $G=E_1$;} \\
   xT(G\setminus e;x,y)  & \hbox{if e is a bridge;} \\
   yT(G\setminus e;x,y)  & \hbox{if e is a loop;} \\
   T(G\setminus e;x,y) + T(G/ e;x,y)  & \hbox{if e is neither a bridge nor a loop.} \\
  \end{array}
\right.$$
\end{defi}

The recursive process to compute the Tutte polynomial in this
second definition is independent on the order in which the edges
are chosen: this can be proven by showing that Definitions
\ref{defspanning} and \ref{contracting} are equivalent
\cite{equivalence}.

Once we have the definition, we can state some of the main
properties of the Tutte polynomial (for more details, see
\cite{bollobas, equivalence, machi}). Recall that a one point join
$G*H$ of two graphs $G$ and $H$ is obtained by identifying a
vertex $v$ of $G$ and a vertex $w$ of $H$ into a single vertex of
$G*H$. The following property can be easily proven by using
Definition \ref{defspanning}:
\begin{eqnarray}\label{prodotto}
T(G*H;x,y)=T(G;x,y)T(H;x,y).
\end{eqnarray}
In the sequel of the paper, we will refer to this equality as
Property \eqref{prodotto}. Next, recall that if $G=(V,E)$ is
connected and for some $W\subset V$ the graph $G \setminus W$ is
disconnected, we say that $W$ separates $G$. A graph $G$ is
\emph{$2$-connected} if either $G$ is the complete graph $K_3$ or
it has at least $4$ vertices and no vertex separates $G$.
\begin{teo}\cite[Theorem 26]{machi}\label{2connected}
If $G$ is a $2$-connected graph, then $T(G;x,y)$ is irreducible in
$\mathbb{Z}[x,y]$.
\end{teo}
It follows that, since the graphs $\Gamma_n$ and $\Sigma_n$ are
$2$-connected, for each $n\geq 1$, their Tutte polynomials are
irreducible in $\mathbb{Z}[x,y]$.

In the next sections, we will be interested in special evaluations
of the Tutte polynomial, that allow us to deduce many
combinatorial and algebraic properties of the graphs considered.
In the following theorem, we collect many of these properties that
are well-known in literature.

\begin{teo} \cite[Theorems 3 and 8]{machi}\label{evaluations}
Let $G=(V,E)$ be a connected graph and denote by $T(G; x,y)$ its
Tutte polynomial. Then:
\begin{enumerate}
\item $T(G; 1,1) =\tau(G)$;
\item $T(G; 1,2)$ is the number of spanning connected subgraphs of $G$;
\item $T(G; 2,1)$ is the number of spanning forests  of $G$;
\item $T(G; 2,2)= 2^{|E|}$;
\item $T(G; 2,0)$ is the number of acyclic orientations of $G$, i.e., orientations having no oriented
cycles.
\end{enumerate}
\end{teo}
Let $R(G,p)$ the \emph{reliability polynomial} of the graph $G$.
For a random model where each edge of $G$ is independently chosen
to be active (or open) with probability $p$ or inactive (closed)
with probability $1-p$, it provides the probability that there is
a path of active edges between each pair of vertices of $G$. Next,
let $\chi(G, \lambda)$ be the \emph{chromatic polynomial} of $G$,
giving, for all values $\lambda$, the number of proper
$\lambda$-colorings of $G$. Finally, let $Z$ be the
\emph{partition function} of the Ising model, which is obtained as
a special case of the $Q$-Potts model on $G$, for $Q=2$.

\noindent The connection with the Tutte polynomial is given by the
following theorem.
\begin{teo}\cite{machi, potts}\label{twopolynomials}
 Let $G=(V,E)$ be a graph. Then,
\begin{enumerate}
\item $\displaystyle
R(G,p)=p^{|V(G)|-1}(1-p)^{|E(G)|-|V(G)|+1}T\left(G;1,\frac{1}{1-p}\right)$;
\item $\chi(G, \lambda) = (-1)^{r(G)}\lambda ^{k(G)}
T(G;1-\lambda,0)$;
\item $\displaystyle
Z = 2(e^{2\beta J}-1)^{|V(G)|-1}e^{-\beta
J|E(G)|}T\left(G;\frac{e^{2\beta J}+1}{e^{2\beta J}-1},e^{2\beta
J}\right)$, where $J$ is a positive constant and $\beta$ is the
\lq\lq inverse temperature\rq\rq.
\end{enumerate}
\end{teo}

\subsection{Groups of automorphisms of rooted regular trees}

Let $T_q$ be the infinite regular rooted tree of degree $q$, i.e.,
the rooted tree in which each vertex has $q$ children. Each vertex
of the $n$-th level of the tree can be regarded as a word of
length $n$ in the alphabet $X=\{0,1,\ldots, q-1\}$. Moreover, one
can identify the set $X^{\omega}$ of infinite words in $X$ with
the set $\partial T_q$ of infinite geodesic rays starting at the
root of $T_q$. Next, let $S<Aut(T_q)$ be a group acting on $T_q$
by automorphisms generated by a finite symmetric set of generators
$Y$. Moreover, suppose that the action is transitive on each level
of the tree.

\begin{defi}\label{defischreiernovembre}
The $n$-th  \emph{Schreier graph} of the action of $S$ on $T_q$,
with respect to the generating set $Y$, is a graph whose vertex
set coincides with the set of vertices of the $n$-th level of the
tree, and two vertices $u,v$ are adjacent if and only if there
exists $s\in Y$ such that $s(u)=v$. If this is the case, the edge
joining $u$ and $v$ is labelled by $s$.
\end{defi}
The vertices of this graph are labelled by words of length $n$ in
$X$ and the edges are labelled by elements of $Y$. The Schreier
graph is thus a regular graph of degree $|Y|$ with $q^n$ vertices,
and it is connected, since the action of $S$ is level-transitive.

\begin{defi}\cite{volo}\label{defiselfsimilar}
A finitely generated group $S<Aut(T_q)$ is  \emph{self-similar} if,
for all $g\in S, x\in X$, there exist $h\in S, y\in X$ such that
$$
g(xw)=yh(w),
$$
for all finite words $w$ in the alphabet $X$.
\end{defi}

Self-similarity implies that $S$ can be embedded into the wreath
product $Sym(q)\wr S$, where $Sym(q)$ denotes the symmetric group
on $q$ elements, so that any automorphism $g\in S$ can be
represented as
$$
g=\tau(g_0,\ldots,g_{q-1}),
$$
where $\tau\in Sym(q)$ describes the action of $g$ on the first
level of $T_q$ and $g_i\in S, i=0,...,q-1$, is the restriction of
$g$ on the full subtree of $T_q$ rooted at the vertex $i$ of the
first level of $T_q$ (observe that any such subtree is isomorphic
to $T_q$). Hence, if $x\in X$ and $w$ is a finite word in $X$, we
have
$$
g(xw)=\tau(x)g_x(w).
$$
\indent The class of self-similar groups contains many interesting
examples of groups which have exotic properties: for instance, the
first Grigorchuk group and the Basilica group. We recall here that
the first Grigorchuk group was the first example of a group of
intermediate growth (see \cite{grigorchuk} for a detailed account
and further references). As regards the Basilica group, it was
introduced by R. Grigorchuk and A. \.{Z}uk in \cite{primo} as a
group generated by a three-state automaton; it is an example of an
amenable group (see \cite{amen}) not belonging to the class of
subexponentially amenable groups and it can be described as an
iterated monodromy group \cite{volo}.\\ \indent In Section
\ref{sectionhanoi}, we will describe the Schreier graphs
$\{\Sigma_n\}_{n\geq 1}$ of the Hanoi Towers group $H^{(3)}$,
using the self-similar representation of its generators. In
Section \ref{4.2}, we will compute the Tutte polynomial of the
graphs $\{\Sigma_n\}_{n\geq 1}$, using their self-similar
structure.

\section{The Tutte polynomial of the Sierpi\'{n}ski graphs}\label{sezione serpin}

In this section we study the Tutte polynomial of a sequence of
graphs $\{\Gamma_n\}_{n\geq 1}$, approximating the famous
Sierpi\'{n}ski gasket. For each $n\geq 1$, the graph $\Gamma_n$ is
very close to the Schreier graph $\Sigma_n$ of the group $H^{(3)}$
considered in the next section. More precisely, one can obtain
$\Gamma_n$ from $\Sigma_n$ by removing loops and contracting, at
each step, all the \emph{special edges} of $\Sigma_n$, joining two
different elementary triangles (see Section \ref{sectionhanoi}).
The graph $\Gamma_n$ is also self-similar in the sense of
\cite{wagner2}, as can be seen in the following picture,
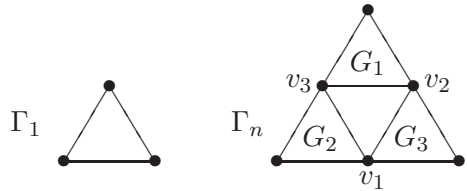
\begin{center}\unitlength=0.2mm
\begin{picture}(400,105)
\put(35,30){$\Gamma_1$}\put(180,30){$\Gamma_n$}

\letvertex A=(100,60)\letvertex B=(70,10)\letvertex C=(130,10)
\letvertex D=(270,110)\letvertex E=(240,60)\letvertex F=(210,10)\letvertex G=(270,10)
\letvertex H=(330,10)\letvertex I=(300,60)

\put(258,70){$G_1$}\put(227,20){$G_2$}\put(287,20){$G_3$}

\drawvertex(A){$\bullet$}\drawvertex(B){$\bullet$}
\drawvertex(C){$\bullet$}\drawvertex(D){$\bullet$}
\drawvertex(E){$\bullet$}\drawvertex(F){$\bullet$}
\drawvertex(G){$\bullet$}\drawvertex(H){$\bullet$}
\drawvertex(I){$\bullet$}

\put(216,59){$v_3$}\put(264,-5){$v_1$}\put(307,59){$v_2$}

\drawundirectededge(B,A){} \drawundirectededge(C,B){}
\drawundirectededge(A,C){} \drawundirectededge(E,D){}
\drawundirectededge(F,E){} \drawundirectededge(G,F){}
\drawundirectededge(H,G){} \drawundirectededge(I,H){}
\drawundirectededge(D,I){} \drawundirectededge(I,E){}
\drawundirectededge(E,G){} \drawundirectededge(G,I){}
\end{picture}
\end{center}
where the subgraphs $G_1, G_2, G_3$ of $\Gamma_n$ are isomorphic
to $\Gamma_{n-1}$, and they are joint at the vertices $v_1, v_2$
and $v_3$, called \emph{special vertices} of $\Gamma_n$. Note that
$v_i\in G_j$, for each $i\neq j$. It is easy to prove by induction
the following equalities:
$$
|V(\Gamma_n)| = \frac{3^n+3}{2} \qquad \qquad |E(\Gamma_n)|=3^n.
$$
We want to compute the Tutte polynomial $T(\Gamma_n;x,y)$ by using
Definition \ref{defspanning}. First of all, we define the
following partition of the set of the spanning subgraphs of
$\Gamma_n$:
\begin{itemize}
\item $D_{2,n}$ denotes the set of spanning subgraphs of $\Gamma_n$,
where the three outmost vertices belong to the same connected
component;
\item $D_{1,n}^u$ denotes the set of spanning subgraphs of $\Gamma_n$, where the leftmost and rightmost
vertices belong to the same connected component, and the upmost
one belongs to a different connected component. Similarly, by
rotation, $D_{1,n}^r$ (respectively $D_{1,n}^l$) denotes the set
of spanning subgraphs of $\Gamma_n$, where the rightmost
(respectively leftmost) vertex is not in the same connected
component containing the two other outmost vertices;
\item $D_{0,n}$ denotes the set of spanning subgraphs of $\Gamma_n$, where the three outmost
vertices belong to three different connected components.
\end{itemize}
To draw a subgraph of $\Gamma_n$ of the previous types, we will
use the following notation.
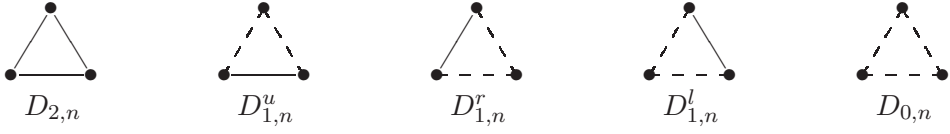
\begin{center}\unitlength=0.35mm
\begin{picture}(400,40)
\letvertex A=(25,35)\letvertex B=(10,10)
\letvertex C=(40,10)

\letvertex a=(105,35)\letvertex b=(90,10)
\letvertex c=(120,10)

\letvertex aa=(185,35)\letvertex bb=(170,10)
\letvertex cc=(200,10)

\letvertex aaa=(265,35)\letvertex bbb=(250,10)
\letvertex ccc=(280,10)

\letvertex AAA=(345,35)\letvertex BBB=(330,10)
\letvertex CCC=(360,10)

\put(15,-5){$D_{2,n}$}\put(95,-5){$D_{1,n}^u$}\put(175,-5){$D_{1,n}^r$}\put(255,-5){$D_{1,n}^l$}\put(335,-5){$D_{0,n}$}

\dashline[0]{4}(90,10)(105,35)
\dashline[0]{4}(120,10)(105,35)

\dashline[0]{4}(200,10)(185,35) \dashline[0]{4}(200,10)(170,10)

\dashline[0]{4}(250,10)(265,35)
\dashline[0]{4}(250,10)(280,10)

\dashline[0]{4}(330,10)(345,35)
\dashline[0]{4}(360,10)(345,35)
\dashline[0]{4}(330,10)(360,10)

\drawundirectededge(A,B){} \drawundirectededge(B,C){}
\drawundirectededge(C,A){}

\drawundirectededge(aa,bb){}

\drawundirectededge(aaa,ccc){}

%\drawundirectededge(a,b){}\drawundirectededge(c,a){}

\drawundirectededge(b,c){}

%\drawundirectededge(AA,BB){} \drawundirectededge(BB,CC){}
%\drawundirectededge(CC,AA){}

\drawvertex(a){$\bullet$}\drawvertex(b){$\bullet$}
\drawvertex(c){$\bullet$}
\drawvertex(A){$\bullet$}\drawvertex(B){$\bullet$}
\drawvertex(C){$\bullet$}
\drawvertex(AAA){$\bullet$}\drawvertex(BBB){$\bullet$}
\drawvertex(CCC){$\bullet$}

\drawvertex(aa){$\bullet$}\drawvertex(bb){$\bullet$}
\drawvertex(cc){$\bullet$}
\drawvertex(aaa){$\bullet$}\drawvertex(bbb){$\bullet$}
\drawvertex(ccc){$\bullet$}
\end{picture}
\end{center}
Hence, with our convention, a black line joining two outmost
vertices in a diagram means that there is a path joining them in
the subgraph and so they are in the same connected component. In
the following pictures, we give explicit examples of spanning
subgraphs of $\Gamma_3$, which are in $D_{2,3}, D^u_{1,3}$ and
$D_{0,3}$, respectively.
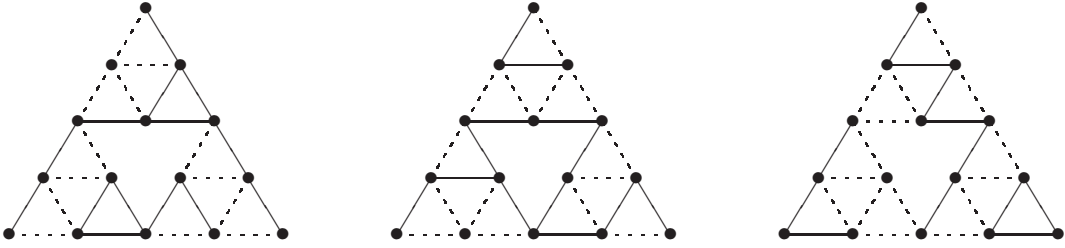
\begin{center}
\begin{picture}(515,115)\unitlength=0.15mm

\letvertex A=(170,210)\letvertex B=(140,160)\letvertex C=(110,110)
\letvertex D=(80,60)\letvertex E=(50,10)\letvertex F=(110,10)\letvertex G=(170,10)
\letvertex H=(230,10)\letvertex I=(290,10)
\letvertex L=(260,60)\letvertex M=(230,110)\letvertex N=(200,160)
\letvertex O=(170,110)\letvertex P=(140,60)\letvertex Q=(200,60)

\drawvertex(A){$\bullet$}\drawvertex(B){$\bullet$}
\drawvertex(C){$\bullet$}\drawvertex(D){$\bullet$}
\drawvertex(E){$\bullet$}\drawvertex(F){$\bullet$}
\drawvertex(G){$\bullet$}\drawvertex(H){$\bullet$}
\drawvertex(I){$\bullet$}\drawvertex(L){$\bullet$}\drawvertex(M){$\bullet$}
\drawvertex(N){$\bullet$}\drawvertex(O){$\bullet$}
\drawvertex(P){$\bullet$}\drawvertex(Q){$\bullet$}

%\drawundirectededge(B,A){}\drawundirectededge(C,B){}\drawundirectededge(N,B){}\drawundirectededge(B,O){}
%\drawundirectededge(C,P){}
%\drawundirectededge(P,D){}\drawundirectededge(D,F){}\drawundirectededge(F,E){}\drawundirectededge(H,G){}
%\drawundirectededge(I,H){}\drawundirectededge(H,L){}\drawundirectededge(L,Q){}\drawundirectededge(Q,M){}

\dashline[0]{4}(140,160)(110,110) \dashline[0]{4}(140,160)(200,160) \dashline[0]{4}(140,160)(170,210)
\dashline[0]{4}(140,160)(170,110) \dashline[0]{4}(110,110)(140,60)\dashline[0]{4}(80,60)(140,60) \dashline[0]{4}(80,60)(110,10) \dashline[0]{4}(110,10)(50,10) \dashline[0]{4}(230,10)(290,10)
\dashline[0]{4}(230,10)(260,60)
\dashline[0]{4}(230,10)(170,10)
\dashline[0]{4}(200,60)(260,60)  \dashline[0]{4}(200,60)(230,110)

\drawundirectededge(E,D){} \drawundirectededge(D,C){}
 \drawundirectededge(A,N){}\drawundirectededge(G,F){}
\drawundirectededge(N,M){} \drawundirectededge(M,L){}
\drawundirectededge(L,I){} \drawundirectededge(O,C){}
\drawundirectededge(M,O){} \drawundirectededge(O,N){}
 \drawundirectededge(P,G){}
 \drawundirectededge(G,Q){}
\drawundirectededge(F,P){} \drawundirectededge(Q,H){}

%%%%%%%%%%%%%%%%%%%%%%%%%%%%%%%%%%%%%%%%%%%%%%%%%%%%%%%%%%%%%%%%%%%%%%
\letvertex AA=(510,210)\letvertex BB=(480,160)\letvertex CC=(450,110)
\letvertex DD=(420,60)\letvertex EE=(390,10)\letvertex FF=(450,10)\letvertex GG=(510,10)
\letvertex HH=(570,10)\letvertex II=(630,10)
\letvertex LL=(600,60)\letvertex MM=(570,110)\letvertex NN=(540,160)
\letvertex OO=(510,110)\letvertex PP=(480,60)\letvertex QQ=(540,60)

\drawvertex(AA){$\bullet$}\drawvertex(BB){$\bullet$}
\drawvertex(CC){$\bullet$}\drawvertex(DD){$\bullet$}
\drawvertex(EE){$\bullet$}\drawvertex(FF){$\bullet$}
\drawvertex(GG){$\bullet$}\drawvertex(HH){$\bullet$}
\drawvertex(II){$\bullet$}\drawvertex(LL){$\bullet$}\drawvertex(MM){$\bullet$}
\drawvertex(NN){$\bullet$}\drawvertex(OO){$\bullet$}
\drawvertex(PP){$\bullet$}\drawvertex(QQ){$\bullet$}

\drawundirectededge(BB,AA){}\drawundirectededge(CC,PP){}\drawundirectededge(OO,CC){}
\drawundirectededge(NN,BB){}\drawundirectededge(EE,DD){}\drawundirectededge(MM,OO){}\drawundirectededge(PP,GG){}
\drawundirectededge(PP,DD){}\drawundirectededge(HH,GG){}\drawundirectededge(HH,LL){}
\drawundirectededge(QQ,MM){}\drawundirectededge(GG,QQ){}\drawundirectededge(LL,II){}

%\drawundirectededge(DD,CC){}\drawundirectededge(CC,BB){} \drawundirectededge(AA,NN){}
%\drawundirectededge(NN,MM){} %\drawundirectededge(MM,LL){}\drawundirectededge(QQ,HH){}
% \drawundirectededge(II,HH){} \drawundirectededge(GG,FF){} %\drawundirectededge(FF,EE){} \drawundirectededge(LL,QQ){} %\drawundirectededge(BB,OO){}
%\drawundirectededge(OO,NN){}
%\drawundirectededge(DD,FF){} \drawundirectededge(FF,PP){}

\dashline[0]{4}(600,60)(540,60)
\dashline[0]{4}(450,10)(390,10)
\dashline[0]{4}(450,10)(510,10)
\dashline[0]{4}(450,10)(420,60)
\dashline[0]{4}(450,10)(480,60)
\dashline[0]{4}(570,110)(600,60)
\dashline[0]{4}(510,110)(480,160)
\dashline[0]{4}(570,10)(540,60)
\dashline[0]{4}(570,10)(630,10)
\dashline[0]{4}(450,110)(420,60)
\dashline[0]{4}(540,160)(510,210)
\dashline[0]{4}(540,160)(570,110)
\dashline[0]{4}(540,160)(510,110)
\dashline[0]{4}(450,110)(480,160)

%%%%%%%%%%%%%%%%%%%%%%%%%%%%%%%%%%%%%%%%%%%%%%%%%%%%%%%%%%%%%%%%%%%%%%%%%%%%%%%%%%%%%%%%%%%%%%%%%%%%%

\letvertex AAA=(850,210)\letvertex BBB=(820,160)\letvertex CCC=(790,110)
\letvertex DDD=(760,60)\letvertex EEE=(730,10)\letvertex FFF=(790,10)\letvertex GGG=(850,10)
\letvertex HHH=(910,10)\letvertex III=(970,10)
\letvertex LLL=(940,60)\letvertex MMM=(910,110)\letvertex NNN=(880,160)
\letvertex OOO=(850,110)\letvertex PPP=(820,60)\letvertex QQQ=(880,60)

\drawvertex(AAA){$\bullet$}\drawvertex(BBB){$\bullet$}
\drawvertex(CCC){$\bullet$}\drawvertex(DDD){$\bullet$}
\drawvertex(EEE){$\bullet$}\drawvertex(FFF){$\bullet$}
\drawvertex(GGG){$\bullet$}\drawvertex(HHH){$\bullet$}
\drawvertex(III){$\bullet$}\drawvertex(LLL){$\bullet$}\drawvertex(MMM){$\bullet$}
\drawvertex(NNN){$\bullet$}\drawvertex(OOO){$\bullet$}
\drawvertex(PPP){$\bullet$}\drawvertex(QQQ){$\bullet$}

\drawundirectededge(BBB,AAA){}\drawundirectededge(MMM,OOO){}\drawundirectededge(QQQ,MMM){}\drawundirectededge(GGG,QQQ){}
\drawundirectededge(NNN,BBB){}\drawundirectededge(OOO,NNN){}\drawundirectededge(HHH,LLL){}\drawundirectededge(LLL,III){}
\drawundirectededge(III,HHH){}\drawundirectededge(FFF,EEE){}
\drawundirectededge(DDD,CCC){}\drawundirectededge(EEE,DDD){}

%\drawundirectededge(CCC,BBB){} \drawundirectededge(AAA,NNN){}
%\drawundirectededge(NNN,MMM){} \drawundirectededge(MMM,LLL){}
%\drawundirectededge(HHH,GGG){} \drawundirectededge(GGG,FFF){}
%\drawundirectededge(OOO,CCC){}\drawundirectededge(PPP,DDD){}
%\drawundirectededge(LLL,QQQ){} \drawundirectededge(BBB,OOO){}
% \drawundirectededge(CCC,PPP){}\drawundirectededge(PPP,GGG){} \drawundirectededge(DDD,FFF){}
%\drawundirectededge(FFF,PPP){} \drawundirectededge(QQQ,HHH){}

%%%%%%%%%%%%%%%%%%%%%%%%%%%%%%%%%%%%%%%%%%%%%%%%%%%%%%%%%%%%%%%%%%%%%%%%%%%%%5

\dashline[0]{4}(790,110)(820,160)
\dashline[0]{4}(790,110)(850,110)
\dashline[0]{4}(790,110)(820,60)
\dashline[0]{4}(820,60)(760,60)
\dashline[0]{4}(820,60)(850,10)
\dashline[0]{4}(820,60)(790,10)
\dashline[0]{4}(790,10)(850,10)
\dashline[0]{4}(790,10) (760,60)

\dashline[0]{4}(850,10)(910,10)

\dashline[0]{4}(910,10)(880,60)
\dashline[0]{4}(850,210) (880,160)
\dashline[0]{4}(880,160)(910,110)
\dashline[0]{4}(910,110)(940,60)
\dashline[0]{4}(940,60)(880,60)
\dashline[0]{4}(820,160)(850,110)
\end{picture}
\end{center}
Observe that, for each $n\geq 1$, we have the partition
$$
D_{2,n} \sqcup D_{1,n}^u\sqcup D_{1,n}^r\sqcup D_{1,n}^l\sqcup
D_{0,n}
$$
of the set of spanning subgraphs of $\Gamma_n$. Next, let us
simply denote by $T_n(x,y)$ the Tutte polynomial $T(\Gamma_n;x,y)$
of $\Gamma_n$ and define, for every $n\geq 1$, the following
polynomials:
\begin{itemize}
\item $\displaystyle T_{2,n}(x,y)= \sum_{A\in D_{2,n}}(x-1)^{r(\Gamma_n)-r(A)}(y-1)^{n(A)}$;
\item $\displaystyle T_{1,n}^u(x,y)= \sum_{A\in D_{1,n}^u}(x-1)^{r(\Gamma_n)-r(A)}(y-1)^{n(A)}$;
%\item $\displaystyle T_{1,n}^r(x,y)= \sum_{A\in D_{1,n}^r}(x-1)^{r(\Gamma_n)-r(A)}(y-1)^{n(A)}$;
%\item $\displaystyle T_{1,n}^l(x,y)= \sum_{A\in D_{1,n}^l}(x-1)^{r(\Gamma_n)-r(A)}(y-1)^{n(A)}$;
\item $\displaystyle T_{0,n}(x,y)= \sum_{A\in D_{0,n}}(x-1)^{r(\Gamma_n)-r(A)}(y-1)^{n(A)}$.
\end{itemize}
Similarly, we define $T_{1,n}^r(x,y)$ and $T_{1,n}^l(x,y)$, by
taking sums over $D_{1,n}^r$ and $D_{1,n}^l$, respectively. Note
that, by the rotational-invariance of the graph $\Gamma_n$, one
has
$$
T_{1,n}^u(x,y) =  T_{1,n}^r(x,y) = T_{1,n}^l(x,y),
$$
so that we can simply use the notation $T_{1,n}(x,y)$ to denote
one of these three polynomials. According with Definition
\ref{defspanning} of the Tutte polynomial, we have:
$$
T_n(x,y) = T_{2,n}(x,y) + 3T_{1,n}(x,y) + T_{0,n}(x,y).
$$
In order to give a recursive formula for $T_n(x,y)$, we provide
recursive formulas for $T_{2,n}(x,y)$, $T_{1,n}(x,y)$ and
$T_{0,n}(x,y)$ (Theorem \ref{ricorsivesierpinski}). For this
purpose, we have to analyze the relation between spanning
subgraphs of $\Gamma_{n+1}$ and spanning subgraphs of $\Gamma_n$.
The following key observation holds.

\emph{There exists a bijection between spanning subgraphs of
$\Gamma_{n+1}$ and spanning subgraphs of the copies $G_1, G_2,
G_3$ of $\Gamma_{n}$ inside $\Gamma_{n+1}$. This bijection is
induced by restrictions.}

Indeed,  the restriction of a spanning subgraph $A$ of
$\Gamma_{n+1}$ to the copies $G_1, G_2$ and $G_3$ of $\Gamma_{n}$
uniquely determines three spanning subgraphs $A_1$, $A_2$ and
$A_3$  of $\Gamma_{n}$; viceversa, given three spanning subgraphs
$A_1$, $A_2$ and $A_3$ of $G_1, G_2$ and $G_3$, respectively, then
their union provides a spanning subgraph $A$ of the whole
$\Gamma_{n+1}$.\\ \indent Therefore, Equation \eqref{defsubgraphs}
in Definition \ref{defspanning} for $\Gamma_{n+1}$ can be
rewritten as
$$
T_{n+1}(x,y) = \sum_{A_i\subseteq G_i, i=1,2,3
}(x-1)^{r(\Gamma_{n+1})-r(A)}(y-1)^{n(A)},
$$
where $A_i$ is the restriction of $A$ to the subgraph $G_i$ of
$\Gamma_{n+1}$.

Next, we need to understand how $r(A)$ and $n(A)$ depend on
$r(A_i)$ and $n(A_i)$, for $i=1,2,3$. Firstly, observe that
$r(\Gamma_{n+1})=3r(\Gamma_{n})-1$ and
$|V(A)|=|V(A_1)|+|V(A_2)|+|V(A_3)|-3$, for every spanning subgraph
$A$ of $\Gamma_{n+1}$. Furthermore, two possibilities can occur.\\
\indent If in the spanning subgraph $A$, obtained by the union of
$A_1,A_2$ and $A_3$, the three special vertices are in the same
connected component, and the two special vertices $v_i,v_j\in A_k$
are in the same connected component of $A_k$ for any $k=1,2,3$,
then  it is easy to check that
$$
k(A) = k(A_1) +k(A_2) + k(A_3) -2 \qquad \mbox{and} \qquad r(A) =
r(A_1) + r(A_2) + r(A_3) -1.
$$
Moreover, one has
\begin{eqnarray*}
n(A)&=&(|E(A_1)|+|E(A_2)|+|E(A_3)|)-(|V(A_1)|+|V(A_2)|+|V(A_3)|-3)\\&+&
(k(A_1) +k(A_2) + k(A_3) -2)\\
&=& n(A_1) + n(A_2) + n(A_3) + 1.
\end{eqnarray*}
Hence, for such a spanning subgraph $A$ of $\Gamma_{n+1}$ (of
\lq\lq first type\rq\rq), one gets:
$$
r(\Gamma_{n+1})-r(A) = \sum_{i=1}^3 (r(\Gamma_{n})-r(A_i)) \qquad
\mbox { and } \qquad n(A)=n(A_1) + n(A_2) + n(A_3) + 1
$$
and so
$$
(x-1)^{r(\Gamma_{n+1})-r(A)}(y-1)^{n(A)}=(y-1) \prod_{i=1}^3
(x-1)^{r(\Gamma_{n})-r(A_i)}(y-1)^{n(A_i)}.
$$
On the other hand, if in the spanning subgraph $A$ obtained by the
union of $A_1,A_2$ and $A_3$ there are two special vertices $v_i,
v_j \in A_k$ which do not belong to the same connected component
of $A_k$, for some $k=1,2,3$ and $i,j\neq k$, it is easy to verify
that
$$
k(A) = k(A_1) +k(A_2) + k(A_3) -3 \qquad \mbox{and} \qquad r(A) =
r(A_1) + r(A_2) + r(A_3).
$$
Moreover, one has in this case
\begin{eqnarray*}
n(A)
&=&(|E(A_1)|+|E(A_2)|+|E(A_3)|)-(|V(A_1)|+|V(A_2)|+|V(A_3)|-3)\\&+&
(k(A_1) +k(A_2) + k(A_3) -3)\\ &=& n(A_1) + n(A_2) + n(A_3).
\end{eqnarray*}
Hence, for such a spanning subgraph $A$ of $\Gamma_{n+1}$ (of
\lq\lq second type\rq\rq), one gets:
$$
r(\Gamma_{n+1})-r(A) = \sum_{i=1}^3 (r(\Gamma_{n})-r(A_i))-1
\qquad \mbox { and } \qquad n(A)=n(A_1) + n(A_2) + n(A_3)
$$
and so
$$
(x-1)^{r(\Gamma_{n+1})-r(A)}(y-1)^{n(A)}=   \frac{1}{(x-1)}
\prod_{i=1}^3   (x-1)^{r(\Gamma_{n})-r(A_i)}(y-1)^{n(A_i)}.
$$

\begin{teo}\label{ricorsivesierpinski}
For each $n\geq 1$, the Tutte polynomial $T_n(x,y)$ of $\Gamma_n$
is given by
$$
T_n(x,y)=T_{2,n}(x,y)+3T_{1,n}(x,y)+T_{0,n}(x,y),
$$
where the polynomials $T_{2,n}(x,y)$, $T_{1,n}(x,y)$,
$T_{0,n}(x,y) \in \mathbb{Z}[x,y]$ satisfy the following recursive
relations:
\begin{eqnarray}\label{t2n}
T_{2,n+1}(x,y) &=& (y-1)T_{2,n}^3 +
\frac{1}{x-1}\left(6T_{2,n}^2T_{1,n}+3T_{2,n}T_{1,n}^2\right)
\end{eqnarray}
\begin{eqnarray}\label{t1n}
T_{1,n+1}(x,y)&=&
(y-1)T_{2,n}^2T_{1,n}+\frac{1}{x-1}\left(T_{2,n}^2T_{0,n}+7T_{2,n}T_{1,n}^2\right.\\&+&\left.2T_{2,n}T_{1,n}T_{0,n}+4T_{1,n}^3
+T_{1,n}^2T_{0,n}\right)\nonumber
\end{eqnarray}
\begin{eqnarray}\label{t0n}
T_{0,n+1}(x,y)&=&
(y-1)\left(3T_{2,n}T_{1,n}^2+T_{1,n}^3\right)+\frac{1}{x-1}\left(12T_{2,n}T_{1,n}T_{0,n}\right.
\\
&+&\left.3T_{2,n}T_{0,n}^2+14T_{1,n}^3+24T_{1,n}^2T_{0,n}
+9T_{1,n}T_{0,n}^2+T_{0,n}^3\right),\nonumber
\end{eqnarray}
with initial conditions
$$
T_{2,1}(x,y)=y+2 \qquad T_{1,1}(x,y)=x-1 \qquad T_{0,1}(x,y)=(x-1)^2.
$$
\end{teo}

\begin{proof}
The initial conditions are easy to be verified. The strategy of
the proof is to study all the possible choices of spanning
subgraphs $A_i$ in the three copies $G_i$ of $\Gamma_n$ inside
$\Gamma_{n+1}$, for $i=1,2,3$, and analyze which kind of
contribution they give to $T_{2,n+1}(x,y)$, $T_{1,n+1}(x,y)$ and
$T_{0,n+1}(x,y)$.

Let us start by studying which configurations of spanning
subgraphs $A_i$, for $i=1,2,3$, give a contribution to
$T_{2,n+1}(x,y)$. First, we have the following configuration.
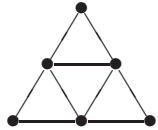
\begin{center}
\begin{picture}(400,55)

\letvertex A=(200,55)\letvertex B=(185,30)
\letvertex C=(170,5)\letvertex D=(200,5)
\letvertex E=(230,5)\letvertex F=(215,30)

\drawundirectededge(A,B){} \drawundirectededge(B,C){}
\drawundirectededge(C,D){} \drawundirectededge(D,E){}
\drawundirectededge(E,F){}\drawundirectededge(F,A){}\drawundirectededge(B,F){}\drawundirectededge(B,D){}\drawundirectededge(D,F){}

\drawvertex(A){$\bullet$}\drawvertex(B){$\bullet$}
\drawvertex(C){$\bullet$}\drawvertex(D){$\bullet$}
\drawvertex(E){$\bullet$}\drawvertex(F){$\bullet$}
\end{picture}
\end{center}
Here, we choose each $A_i$ in $D_{2,n}$; this contributes by a
term $(y-1)T_{2,n}^3$, since in this subgraph the three special
vertices are connected (first type). Moreover, we have two other
possibilities, concerning spanning subgraphs of the second type,
represented in the following two pictures.
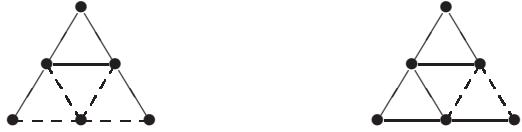
\begin{center}
\begin{picture}(400,60)
\letvertex A=(120,60)\letvertex B=(105,35)
\letvertex C=(90,10)\letvertex D=(120,10)
\letvertex E=(150,10)\letvertex F=(135,35)

\letvertex a=(280,60)\letvertex b=(265,35)
\letvertex c=(250,10)\letvertex d=(280,10)
\letvertex e=(310,10)\letvertex f=(295,35)

\drawundirectededge(A,B){}\drawundirectededge(B,F){}\drawundirectededge(F,A){}
\drawundirectededge(E,F){}\drawundirectededge(B,C){}

%\thinlines

\drawundirectededge(a,b){} \drawundirectededge(b,c){}
\drawundirectededge(c,d){} \drawundirectededge(d,e){}
\drawundirectededge(f,a){}\drawundirectededge(b,f){}\drawundirectededge(b,d){}
%\thinlines

\drawvertex(a){$\bullet$}\drawvertex(b){$\bullet$}
\drawvertex(c){$\bullet$}\drawvertex(d){$\bullet$}
\drawvertex(e){$\bullet$}\drawvertex(f){$\bullet$}
\drawvertex(A){$\bullet$}\drawvertex(B){$\bullet$}
\drawvertex(C){$\bullet$}\drawvertex(D){$\bullet$}
\drawvertex(E){$\bullet$}\drawvertex(F){$\bullet$}

\dashline[0]{5}(280,10)(295,35) \dashline[0]{5}(295,35)(310,10)

\dashline[0]{5}(90,10)(120,10)

\dashline[0]{5}(120,10)(105,35)

\dashline[0]{5}(120,10)(150,10)

\dashline[0]{5}(120,10)(135,35)

%\drawundirectededge(C,D) \drawundirectededge(D,E){}
%\drawundirectededge(B,D){}\drawundirectededge(D,F){}\drawundirectededge(d,f){}\drawundirectededge(e,f){}

\end{picture}
\end{center}
The first configuration contributes by a term
$\frac{3}{x-1}T_{2,n}T_{1,n}^2$ (we have to consider the possible
rotations); the second one contributes by
$\frac{6}{x-1}T_{2,n}^2T_{1,n}$ (we have to take into account all
the possible symmetries of $\Gamma_{n+1}$). This gives Equation
\eqref{t2n}.

Next, we study  the contributions to the polynomial
$T_{1,n+1}(x,y)$. The only case of subgraph of the first type is
represented in the following picture.

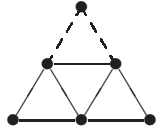
\begin{center}
\begin{picture}(400,60)
\letvertex A=(200,60)\letvertex B=(185,35)
\letvertex C=(170,10)\letvertex D=(200,10)
\letvertex E=(230,10)\letvertex F=(215,35)

 \drawundirectededge(B,C){}
\drawundirectededge(C,D){} \drawundirectededge(D,E){}
\drawundirectededge(E,F){}\drawundirectededge(B,F){}\drawundirectededge(B,D){}\drawundirectededge(D,F){}

\dashline[0]{5}(200,60)(185,35) \dashline[0]{5}(215,35)(200,60)

%\drawundirectededge(A,B){}\drawundirectededge(F,A){}

\drawvertex(A){$\bullet$}\drawvertex(B){$\bullet$}
\drawvertex(C){$\bullet$}\drawvertex(D){$\bullet$}
\drawvertex(E){$\bullet$}\drawvertex(F){$\bullet$}
\end{picture}
\end{center}
It gives a contribution of $(y-1)T_{2,n}^2T_{1,n}$. Then, consider
the following configurations.
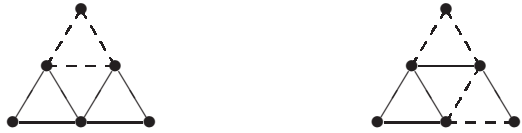
\begin{center}
\begin{picture}(400,60)
\letvertex A=(120,60)\letvertex B=(105,35)
\letvertex C=(90,10)\letvertex D=(120,10)
\letvertex E=(150,10)\letvertex F=(135,35)

\letvertex a=(280,60)\letvertex b=(265,35)
\letvertex c=(250,10)\letvertex d=(280,10)
\letvertex e=(310,10)\letvertex f=(295,35)

\drawundirectededge(B,C){} \drawundirectededge(C,D){}
\drawundirectededge(D,E){}
\drawundirectededge(E,F){}\drawundirectededge(B,D){}\drawundirectededge(D,F){}

%\drawundirectededge(A,B){}\drawundirectededge(F,A){}\drawundirectededge(B,F){}
\dashline[0]{5}(120,60)(105,35) \dashline[0]{5}(135,35)(120,60)
\dashline[0]{5}(105,35)(135,35)

%\drawundirectededge(a,b){}\drawundirectededge(f,a){}\drawundirectededge(d,e){}\drawundirectededge(d,f){}

\dashline[0]{5}(280,60)(265,35) \dashline[0]{5}(280,60)(295,35)
\dashline[0]{5}(280,10)(310,10)\dashline[0]{5}(280,10)(295,35)

\drawundirectededge(b,c){} \drawundirectededge(c,d){}
\drawundirectededge(e,f){}\drawundirectededge(b,f){}\drawundirectededge(b,d){}

\drawvertex(a){$\bullet$}\drawvertex(b){$\bullet$}
\drawvertex(c){$\bullet$}\drawvertex(d){$\bullet$}
\drawvertex(e){$\bullet$}\drawvertex(f){$\bullet$}
\drawvertex(A){$\bullet$}\drawvertex(B){$\bullet$}
\drawvertex(C){$\bullet$}\drawvertex(D){$\bullet$}
\drawvertex(E){$\bullet$}\drawvertex(F){$\bullet$}
\end{picture}
\end{center}
The first one contributes by $\frac{1}{x-1}T_{2,n}^2T_{0,n}$; the
second one gives the contribution $\frac{2}{x-1}T_{2,n}T_{1,n}^2$
(we have to take into account one reflection with respect to the
vertical axis). Let us analyze now the three following
configurations.
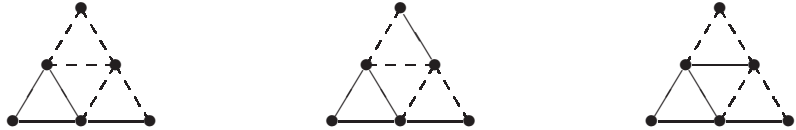
\begin{center}
\begin{picture}(400,60)
\letvertex A=(60,60)\letvertex B=(45,35)
\letvertex C=(30,10)\letvertex D=(60,10)
\letvertex E=(90,10)\letvertex F=(75,35)
\letvertex a=(200,60)\letvertex b=(185,35)
\letvertex c=(170,10)\letvertex d=(200,10)
\letvertex e=(230,10)\letvertex f=(215,35)
\letvertex AA=(340,60)\letvertex BB=(325,35)
\letvertex CC=(310,10)\letvertex DD=(340,10)
\letvertex EE=(370,10)\letvertex FF=(355,35)
\drawundirectededge(B,C){} \drawundirectededge(C,D){}
\drawundirectededge(D,E){} \drawundirectededge(B,D){}
%\drawundirectededge(A,B){}\drawundirectededge(F,A){}\drawundirectededge(B,F){}\drawundirectededge(E,F){}\drawundirectededge(D,F){}
\dashline[0]{5}(60,60)(45,35) \dashline[0]{5}(75,35)(60,60)
\dashline[0]{5}(45,35)(75,35)\dashline[0]{5}(90,10)(75,35)
\dashline[0]{5}(75,35)(60,10)

%\drawundirectededge(a,b){}\drawundirectededge(b,f){}\drawundirectededge(d,f){}\drawundirectededge(e,f){}
\dashline[0]{5}(200,60)(185,35) \dashline[0]{5}(185,35)(215,35)
\dashline[0]{5}(200,10)(215,35)\dashline[0]{5}(230,10)(215,35)
\drawundirectededge(b,c){}
\drawundirectededge(c,d){}\drawundirectededge(f,a){}
\drawundirectededge(d,e){}\drawundirectededge(b,d){}

%\drawundirectededge(AA,BB){}\drawundirectededge(DD,FF){}\drawundirectededge(FF,AA){}\drawundirectededge(EE,FF){}
\dashline[0]{5}(340,60)(325,35) \dashline[0]{5}(355,35)(340,10)
\dashline[0]{5}(340,60)(355,35)\dashline[0]{5}(370,10)(355,35)

\drawundirectededge(BB,CC){}\drawundirectededge(DD,EE){}
\drawundirectededge(CC,DD){}\drawundirectededge(BB,FF){}
\drawundirectededge(BB,DD){}

\drawvertex(a){$\bullet$}\drawvertex(b){$\bullet$}
\drawvertex(c){$\bullet$}\drawvertex(d){$\bullet$}
\drawvertex(e){$\bullet$}\drawvertex(f){$\bullet$}
\drawvertex(A){$\bullet$}\drawvertex(B){$\bullet$}
\drawvertex(C){$\bullet$}\drawvertex(D){$\bullet$}
\drawvertex(E){$\bullet$}\drawvertex(F){$\bullet$}
\drawvertex(AA){$\bullet$}\drawvertex(BB){$\bullet$}
\drawvertex(CC){$\bullet$}\drawvertex(DD){$\bullet$}
\drawvertex(EE){$\bullet$}\drawvertex(FF){$\bullet$}
\end{picture}
\end{center}
The first one gives the contribution
$\frac{1}{x-1}T_{2,n}T_{1,n}T_{0,n}$; both the second and the
third one give $\frac{1}{x-1}T_{2,n}T_{1,n}^2$. All these terms
have to be multiplied by $2$ because of the possible reflections.
The three following configurations give the terms
$\frac{1}{x-1}T_{1,n}^3$, $\frac{1}{x-1}T_{1,n}^2T_{0,n}$ and
$\frac{1}{x-1}T_{2,n}T_{1,n}^2$, respectively.
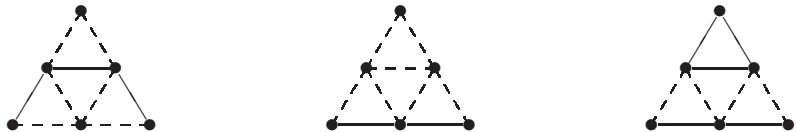
\begin{center}
\begin{picture}(400,60)
\letvertex A=(60,60)\letvertex B=(45,35)
\letvertex C=(30,10)\letvertex D=(60,10)
\letvertex E=(90,10)\letvertex F=(75,35)

\letvertex a=(200,60)\letvertex b=(185,35)
\letvertex c=(170,10)\letvertex d=(200,10)
\letvertex e=(230,10)\letvertex f=(215,35)

\letvertex AA=(340,60)\letvertex BB=(325,35)
\letvertex CC=(310,10)\letvertex DD=(340,10)
\letvertex EE=(370,10)\letvertex FF=(355,35)

\drawundirectededge(B,C){}\drawundirectededge(E,F){}\drawundirectededge(B,F){}

%\drawundirectededge(A,B){}\drawundirectededge(F,A){}
%\drawundirectededge(C,D){}\drawundirectededge(B,D){}\drawundirectededge(D,F){}\drawundirectededge(D,E){}
%\drawundirectededge(a,b){} \drawundirectededge(b,c){}
%\drawundirectededge(e,f){}\drawundirectededge(f,a){}\drawundirectededge(b,f){}\drawundirectededge(b,d){}\drawundirectededge(d,f){}
\dashline[0]{5}(60,60)(45,35) \dashline[0]{5}(75,35)(60,60)
\dashline[0]{5}(30,10)(60,10)\dashline[0]{5}(45,35)(60,10)
\dashline[0]{5}(75,35)(60,10)\dashline[0]{5}(60,10)(90,10)
\dashline[0]{5}(200,60)(185,35)\dashline[0]{5}(170,10)(185,35)
\dashline[0]{5}(215,35)(230,10)
\dashline[0]{5}(200,60)(215,35)\dashline[0]{5}(185,35)(215,35)
\dashline[0]{5}(185,35)(200,10) \dashline[0]{5}(200,10)(215,35)

\drawundirectededge(c,d){} \drawundirectededge(d,e){}

%\drawundirectededge(BB,CC){}\drawundirectededge(BB,DD){}\drawundirectededge(DD,FF){}\drawundirectededge(EE,FF){}
\dashline[0]{5}(310,10)(325,35) \dashline[0]{5}(325,35)(340,10)
\dashline[0]{5}(340,10)(355,35)\dashline[0]{5}(370,10)(355,35)

\drawundirectededge(AA,BB){} \drawundirectededge(CC,DD){}
\drawundirectededge(DD,EE){}
\drawundirectededge(FF,AA){}\drawundirectededge(BB,FF){}

\drawvertex(a){$\bullet$}\drawvertex(b){$\bullet$}
\drawvertex(c){$\bullet$}\drawvertex(d){$\bullet$}
\drawvertex(e){$\bullet$}\drawvertex(f){$\bullet$}
\drawvertex(A){$\bullet$}\drawvertex(B){$\bullet$}
\drawvertex(C){$\bullet$}\drawvertex(D){$\bullet$}
\drawvertex(E){$\bullet$}\drawvertex(F){$\bullet$}
\drawvertex(AA){$\bullet$}\drawvertex(BB){$\bullet$}
\drawvertex(CC){$\bullet$}\drawvertex(DD){$\bullet$}
\drawvertex(EE){$\bullet$}\drawvertex(FF){$\bullet$}
\end{picture}
\end{center}
Finally, the two following configurations give
$\frac{2}{x-1}T_{1,n}^3$ and $\frac{1}{x-1}T_{1,n}^3$,
respectively (for the picture on the left, we have to take into
account a possible reflection).

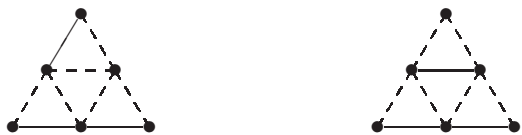
\begin{center}
\begin{picture}(400,60)
\letvertex A=(120,60)\letvertex B=(105,35)
\letvertex C=(90,10)\letvertex D=(120,10)
\letvertex E=(150,10)\letvertex F=(135,35)

\letvertex a=(280,60)\letvertex b=(265,35)
\letvertex c=(250,10)\letvertex d=(280,10)
\letvertex e=(310,10)\letvertex f=(295,35)

% \drawundirectededge(B,C){}\drawundirectededge(E,F){}\drawundirectededge(F,A){}
%\drawundirectededge(B,F){}\drawundirectededge(B,D){}\drawundirectededge(D,F){}
\drawundirectededge(C,D){}
\drawundirectededge(D,E){}\drawundirectededge(A,B){}

\dashline[0]{5}(90,10)(105,35) \dashline[0]{5}(135,35)(150,10)
\dashline[0]{5}(135,35)(120,60)\dashline[0]{5}(135,35)(105,35)
\dashline[0]{5}(105,35)(120,10) \dashline[0]{5}(120,10)(135,35)

%\drawundirectededge(a,b){} \drawundirectededge(b,c){}
%\drawundirectededge(e,f){}\drawundirectededge(f,a){}\drawundirectededge(b,d){}\drawundirectededge(d,f){}
\dashline[0]{5}(280,60)(265,35) \dashline[0]{5}(265,35)(250,10)
\dashline[0]{5}(310,10)(295,35)\dashline[0]{5}(280,60)(295,35)
\dashline[0]{5}(265,35)(280,10) \dashline[0]{5}(280,10)(295,35)

\drawundirectededge(c,d){}
\drawundirectededge(d,e){}\drawundirectededge(b,f){}

\drawvertex(a){$\bullet$}\drawvertex(b){$\bullet$}
\drawvertex(c){$\bullet$}\drawvertex(d){$\bullet$}
\drawvertex(e){$\bullet$}\drawvertex(f){$\bullet$}
\drawvertex(A){$\bullet$}\drawvertex(B){$\bullet$}
\drawvertex(C){$\bullet$}\drawvertex(D){$\bullet$}
\drawvertex(E){$\bullet$}\drawvertex(F){$\bullet$}
\end{picture}
\end{center}
This completes the proof of Equation \eqref{t1n}. Next, let us
consider the contributions to the polynomial $T_{0,n+1}(x,y)$. The
following pictures represent the only cases producing subgraphs of
the first type.
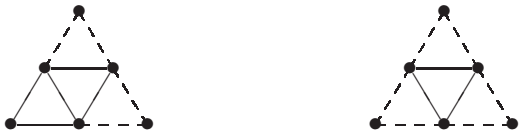
\begin{center}
\begin{picture}(400,60)
\letvertex A=(120,60)\letvertex B=(105,35)
\letvertex C=(90,10)\letvertex D=(120,10)
\letvertex E=(150,10)\letvertex F=(135,35)

\letvertex a=(280,60)\letvertex b=(265,35)
\letvertex c=(250,10)\letvertex d=(280,10)
\letvertex e=(310,10)\letvertex f=(295,35)

%\drawundirectededge(A,B){}\drawundirectededge(F,A){}\drawundirectededge(D,E){}
%\drawundirectededge(E,F){}

\dashline[0]{5}(120,60)(105,35) \dashline[0]{5}(135,35)(120,60)
\dashline[0]{5}(120,10)(150,10)\dashline[0]{5}(150,10)(135,35)

\drawundirectededge(B,C){} \drawundirectededge(C,D){}
\drawundirectededge(B,F){}\drawundirectededge(B,D){}\drawundirectededge(D,F){}

%\drawundirectededge(a,b){}\drawundirectededge(f,a){} \drawundirectededge(d,e){}\drawundirectededge(e,f){}
%\drawundirectededge(b,c){}\drawundirectededge(c,d){}

\dashline[0]{5}(280,60)(265,35) \dashline[0]{5}(295,35)(280,60)
\dashline[0]{5}(280,10)(310,10)\dashline[0]{5}(310,10)(295,35)
\dashline[0]{5}(265,35)(250,10) \dashline[0]{5}(250,10)(280,10)

\drawundirectededge(b,f){}\drawundirectededge(b,d){}\drawundirectededge(d,f){}

\drawvertex(a){$\bullet$}\drawvertex(b){$\bullet$}
\drawvertex(c){$\bullet$}\drawvertex(d){$\bullet$}
\drawvertex(e){$\bullet$}\drawvertex(f){$\bullet$}
\drawvertex(A){$\bullet$}\drawvertex(B){$\bullet$}
\drawvertex(C){$\bullet$}\drawvertex(D){$\bullet$}
\drawvertex(E){$\bullet$}\drawvertex(F){$\bullet$}
\end{picture}
\end{center}
These configurations give the term $(y-1)\left(3T_{2,n}T_{1,n}^2 +
T_{1,n}^3\right)$. Next, look at the following pictures.

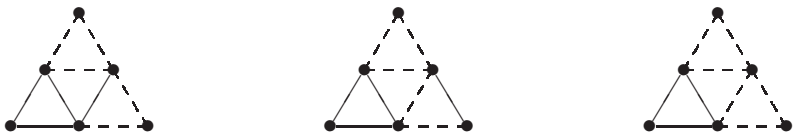
\begin{center}
\begin{picture}(400,60)
\letvertex A=(60,60)\letvertex B=(45,35)
\letvertex C=(30,10)\letvertex D=(60,10)
\letvertex E=(90,10)\letvertex F=(75,35)

\letvertex a=(200,60)\letvertex b=(185,35)
\letvertex c=(170,10)\letvertex d=(200,10)
\letvertex e=(230,10)\letvertex f=(215,35)

\letvertex AA=(340,60)\letvertex BB=(325,35)
\letvertex CC=(310,10)\letvertex DD=(340,10)
\letvertex EE=(370,10)\letvertex FF=(355,35)

%\drawundirectededge(A,B){} \drawundirectededge(D,E){}
%\drawundirectededge(E,F){}\drawundirectededge(F,A){}\drawundirectededge(B,F){}
\dashline[0]{5}(60,60)(45,35) \dashline[0]{5}(60,10)(90,10)
\dashline[0]{5}(90,10)(75,35)\dashline[0]{5}(60,60)(75,35)
\dashline[0]{5}(45,35)(75,35)

\drawundirectededge(B,C){}
\drawundirectededge(C,D){}\drawundirectededge(B,D){}\drawundirectededge(D,F){}
%\drawundirectededge(a,b){} \drawundirectededge(d,e){}
%\drawundirectededge(f,a){}\drawundirectededge(b,f){}\drawundirectededge(d,f){}
\dashline[0]{5}(200,60)(185,35) \dashline[0]{5}(200,10)(230,10)
\dashline[0]{5}(215,35)(200,60)\dashline[0]{5}(215,35)(185,35)
\dashline[0]{5}(215,35)(200,10)

\drawundirectededge(b,c){}
\drawundirectededge(c,d){}\drawundirectededge(b,d){}\drawundirectededge(e,f){}

%\drawundirectededge(AA,BB){} \drawundirectededge(DD,EE){}
%\drawundirectededge(EE,FF){}\drawundirectededge(FF,AA){}\drawundirectededge(BB,FF){}\drawundirectededge(DD,FF){}
\dashline[0]{5}(340,60)(325,35) \dashline[0]{5}(340,10)(370,10)
\dashline[0]{5}(370,10)(355,35)\dashline[0]{5}(340,60)(355,35)
\dashline[0]{5}(325,35)(355,35)\dashline[0]{5}(340,10)(355,35)

\drawundirectededge(BB,CC){}
\drawundirectededge(CC,DD){}\drawundirectededge(BB,DD){}

\drawvertex(a){$\bullet$}\drawvertex(b){$\bullet$}
\drawvertex(c){$\bullet$}\drawvertex(d){$\bullet$}
\drawvertex(e){$\bullet$}\drawvertex(f){$\bullet$}
\drawvertex(A){$\bullet$}\drawvertex(B){$\bullet$}
\drawvertex(C){$\bullet$}\drawvertex(D){$\bullet$}
\drawvertex(E){$\bullet$}\drawvertex(F){$\bullet$}
\drawvertex(AA){$\bullet$}\drawvertex(BB){$\bullet$}
\drawvertex(CC){$\bullet$}\drawvertex(DD){$\bullet$}
\drawvertex(EE){$\bullet$}\drawvertex(FF){$\bullet$}
\end{picture}
\end{center}

These configurations produce the terms
$\frac{6}{x-1}T_{2,n}T_{1,n}T_{0,n}$ (reflections and rotations),
then still $\frac{6}{x-1}T_{2,n}T_{1,n}T_{0,n}$ (reflections and
rotations) and $\frac{3}{x-1}T_{2,n}T_{0,n}^2$ (only rotations),
respectively. The three following configurations give the
contribution $\frac{1}{x-1}\left(6T_{1,n}^3 +
6T_{1,n}^3+6T_{1,n}^2T_{0,n}\right)$.
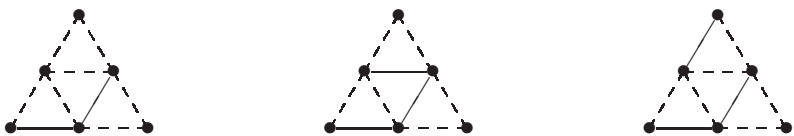
\begin{center}
\begin{picture}(400,60)
\letvertex A=(60,60)\letvertex B=(45,35)
\letvertex C=(30,10)\letvertex D=(60,10)
\letvertex E=(90,10)\letvertex F=(75,35)

\letvertex a=(200,60)\letvertex b=(185,35)
\letvertex c=(170,10)\letvertex d=(200,10)
\letvertex e=(230,10)\letvertex f=(215,35)

\letvertex AA=(340,60)\letvertex BB=(325,35)
\letvertex CC=(310,10)\letvertex DD=(340,10)
\letvertex EE=(370,10)\letvertex FF=(355,35)
%\drawundirectededge(A,B){} \drawundirectededge(B,C){}\drawundirectededge(D,E){}
%\drawundirectededge(E,F){}\drawundirectededge(F,A){}\drawundirectededge(B,F){}\drawundirectededge(B,D){}
\dashline[0]{5}(60,60)(45,35) \dashline[0]{5}(45,35)(30,10)
\dashline[0]{5}(60,10)(90,10)\dashline[0]{5}(90,10)(75,35)
\dashline[0]{5}(75,35)(60,60)
\dashline[0]{5}(45,35)(75,35)\dashline[0]{5}(60,10)(45,35)
\drawundirectededge(C,D){}\drawundirectededge(D,F){}
\drawundirectededge(c,d){}\drawundirectededge(d,f){}\drawundirectededge(b,f){}
%\drawundirectededge(a,b){} \drawundirectededge(b,c){}\drawundirectededge(d,e){}
%\drawundirectededge(e,f){}\drawundirectededge(f,a){}\drawundirectededge(b,d){}
\dashline[0]{5}(200,60)(185,35) \dashline[0]{5}(185,35)(170,10)
\dashline[0]{5}(200,10)(230,10)\dashline[0]{5}(230,10)(215,35)
\dashline[0]{5}(215,35)(200,60) \dashline[0]{5}(185,35)(200,10)

%\drawundirectededge(BB,CC){}\drawundirectededge(DD,EE){}\drawundirectededge(EE,FF){}
%\drawundirectededge(FF,AA){}\drawundirectededge(BB,FF){}\drawundirectededge(BB,DD){}
\dashline[0]{5}(310,10)(325,35)\dashline[0]{5}(340,10)(370,10)
\dashline[0]{5}(370,10)(355,35)\dashline[0]{5}(340,60)(355,35)
\dashline[0]{5}(325,35)(355,35)\dashline[0]{5}(325,35)(340,10)

\drawundirectededge(CC,DD){}\drawundirectededge(DD,FF){}\drawundirectededge(AA,BB){}

\drawvertex(a){$\bullet$}\drawvertex(b){$\bullet$}
\drawvertex(c){$\bullet$}\drawvertex(d){$\bullet$}
\drawvertex(e){$\bullet$}\drawvertex(f){$\bullet$}
\drawvertex(A){$\bullet$}\drawvertex(B){$\bullet$}
\drawvertex(C){$\bullet$}\drawvertex(D){$\bullet$}
\drawvertex(E){$\bullet$}\drawvertex(F){$\bullet$}
\drawvertex(AA){$\bullet$}\drawvertex(BB){$\bullet$}
\drawvertex(CC){$\bullet$}\drawvertex(DD){$\bullet$}
\drawvertex(EE){$\bullet$}\drawvertex(FF){$\bullet$}
\end{picture}
\end{center}
Next, look at the three following pictures. The first one
contributes by $\frac{6}{x-1}T_{1,n}^2T_{0,n}$, the second one by
$\frac{2}{x-1}T_{1,n}^3$ (we only have to take into account one
reflection this time) and the third one by
$\frac{6}{x-1}T_{1,n}^2T_{0,n}$.

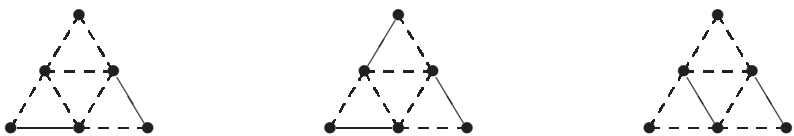
\begin{center}
\begin{picture}(400,60)
\letvertex A=(60,60)\letvertex B=(45,35)
\letvertex C=(30,10)\letvertex D=(60,10)
\letvertex E=(90,10)\letvertex F=(75,35)

\letvertex a=(200,60)\letvertex b=(185,35)
\letvertex c=(170,10)\letvertex d=(200,10)
\letvertex e=(230,10)\letvertex f=(215,35)

\letvertex AA=(340,60)\letvertex BB=(325,35)
\letvertex CC=(310,10)\letvertex DD=(340,10)
\letvertex EE=(370,10)\letvertex FF=(355,35)

%\drawundirectededge(A,B){} \drawundirectededge(B,C){}\drawundirectededge(D,E){}
%\drawundirectededge(F,A){}\drawundirectededge(B,F){}\drawundirectededge(B,D){}\drawundirectededge(D,F){}
\dashline[0]{5}(60,60)(45,35) \dashline[0]{5}(45,35)(30,10)
\dashline[0]{5}(60,10)(90,10)\dashline[0]{5}(60,60)(75,35)
\dashline[0]{5}(45,35)(75,35)\dashline[0]{5}(60,10)(45,35)\dashline[0]{5}(60,10)(75,35)

\drawundirectededge(C,D){} \drawundirectededge(E,F){}
\drawundirectededge(a,b){}\drawundirectededge(c,d){}\drawundirectededge(e,f){}

%\drawundirectededge(b,c){}\drawundirectededge(d,e){}
%\drawundirectededge(f,a){}\drawundirectededge(b,f){}\drawundirectededge(b,d){}\drawundirectededge(d,f){}
\dashline[0]{5}(170,10)(185,35)\dashline[0]{5}(200,10)(230,10)
\dashline[0]{5}(200,60)(215,35)\dashline[0]{5}(215,35)(185,35)
\dashline[0]{5}(185,35)(200,10)\dashline[0]{5}(200,10)(215,35)
\drawundirectededge(BB,DD){}\drawundirectededge(EE,FF){}

%\drawundirectededge(AA,BB){} \drawundirectededge(BB,CC){}
%\drawundirectededge(CC,DD){} \drawundirectededge(DD,EE){}
%\drawundirectededge(FF,AA){}\drawundirectededge(BB,FF){}\drawundirectededge(DD,FF){}
\dashline[0]{5}(340,60)(325,35)\dashline[0]{5}(325,35)(310,10)
\dashline[0]{5}(310,10)(340,10)\dashline[0]{5}(340,10)(370,10)
\dashline[0]{5}(355,35)(340,60)\dashline[0]{5}(325,35)(355,35)\dashline[0]{5}(340,10)(355,35)

\drawvertex(a){$\bullet$}\drawvertex(b){$\bullet$}
\drawvertex(c){$\bullet$}\drawvertex(d){$\bullet$}
\drawvertex(e){$\bullet$}\drawvertex(f){$\bullet$}
\drawvertex(A){$\bullet$}\drawvertex(B){$\bullet$}
\drawvertex(C){$\bullet$}\drawvertex(D){$\bullet$}
\drawvertex(E){$\bullet$}\drawvertex(F){$\bullet$}
\drawvertex(AA){$\bullet$}\drawvertex(BB){$\bullet$}
\drawvertex(CC){$\bullet$}\drawvertex(DD){$\bullet$}
\drawvertex(EE){$\bullet$}\drawvertex(FF){$\bullet$}
\end{picture}
\end{center}
Each one of the following configurations produce a term
$\frac{3}{x-1}T_{1,n}^2T_{0,n}$, since one has to consider
rotations.
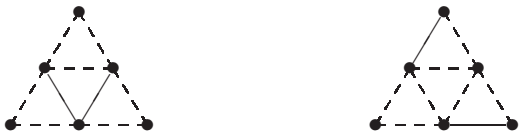
\begin{center}
\begin{picture}(400,60)
\letvertex A=(120,60)\letvertex B=(105,35)
\letvertex C=(90,10)\letvertex D=(120,10)
\letvertex E=(150,10)\letvertex F=(135,35)

\letvertex a=(280,60)\letvertex b=(265,35)
\letvertex c=(250,10)\letvertex d=(280,10)
\letvertex e=(310,10)\letvertex f=(295,35)

\drawundirectededge(B,D){} \drawundirectededge(D,F){}
%\drawundirectededge(A,B){} \drawundirectededge(B,C){}
%\drawundirectededge(C,D){} \drawundirectededge(D,E){}
%\drawundirectededge(E,F){}\drawundirectededge(F,A){}\drawundirectededge(B,F){}
\dashline[0]{5}(120,60)(105,35) \dashline[0]{5}(105,35)(90,10)
\dashline[0]{5}(90,10)(120,10)\dashline[0]{5}(120,10)(150,10)
\dashline[0]{5}(135,35)(150,10)
\dashline[0]{5}(120,60)(135,35)\dashline[0]{5}(135,35)(105,35)

\drawundirectededge(d,e){}\drawundirectededge(a,b){}

%\drawundirectededge(b,c){}\drawundirectededge(c,d){}
%\drawundirectededge(e,f){}\drawundirectededge(f,a){}\drawundirectededge(b,f){}\drawundirectededge(b,d){}\drawundirectededge(d,f){}
\dashline[0]{5}(250,10)(265,35) \dashline[0]{5}(250,10)(280,10)
\dashline[0]{5}(310,10)(295,35)\dashline[0]{5}(280,60)(295,35)
\dashline[0]{5}(265,35)(295,35)
\dashline[0]{5}(265,35)(280,10)\dashline[0]{5}(280,10)(295,35)

\drawvertex(a){$\bullet$}\drawvertex(b){$\bullet$}
\drawvertex(c){$\bullet$}\drawvertex(d){$\bullet$}
\drawvertex(e){$\bullet$}\drawvertex(f){$\bullet$}
\drawvertex(A){$\bullet$}\drawvertex(B){$\bullet$}
\drawvertex(C){$\bullet$}\drawvertex(D){$\bullet$}
\drawvertex(E){$\bullet$}\drawvertex(F){$\bullet$}
\end{picture}
\end{center}

Finally, the three following pictures give
$\frac{3}{x-1}T_{1,n}T_{0,n}^2$, $\frac{6}{x-1}T_{1,n}T_{0,n}^2$
and $\frac{1}{x-1}T_{0,n}^3$, respectively.
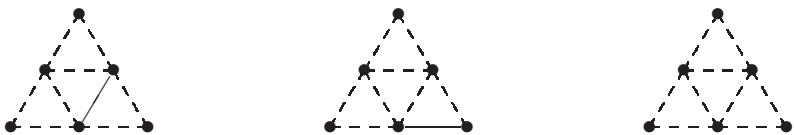
\begin{center}
\begin{picture}(400,60)
\letvertex A=(60,60)\letvertex B=(45,35)\letvertex C=(30,10)\letvertex D=(60,10)\letvertex E=(90,10)\letvertex F=(75,35)

\letvertex a=(200,60)\letvertex b=(185,35)\letvertex c=(170,10)\letvertex d=(200,10)\letvertex e=(230,10)\letvertex f=(215,35)

\letvertex AA=(340,60)\letvertex BB=(325,35)\letvertex CC=(310,10)\letvertex DD=(340,10)\letvertex EE=(370,10)\letvertex FF=(355,35)

%\drawundirectededge(A,B){} \drawundirectededge(B,C){}\drawundirectededge(C,D){} \drawundirectededge(D,E){}
%\drawundirectededge(E,F){}\drawundirectededge(F,A){}\drawundirectededge(B,F){}\drawundirectededge(B,D){}
\drawundirectededge(D,F){}
\dashline[0]{5}(60,60)(45,35)\dashline[0]{5}(45,35)(30,10)\dashline[0]{5}(30,10)(60,10)\dashline[0]{5}(60,10)(90,10)
\dashline[0]{5}(75,35)(90,10)\dashline[0]{5}(75,35)(60,60)\dashline[0]{5}(45,35)(75,35)\dashline[0]{5}(45,35)(60,10)

%\drawundirectededge(a,b){}\drawundirectededge(b,c){}\drawundirectededge(c,d){}\drawundirectededge(e,f){}
%\drawundirectededge(f,a){}\drawundirectededge(b,f){}\drawundirectededge(b,d){}\drawundirectededge(d,f){}
\drawundirectededge(d,e){}
\dashline[0]{5}(200,60)(185,35)\dashline[0]{5}(185,35)(170,10)\dashline[0]{5}(170,10)(200,10)\dashline[0]{5}(230,10)(215,35)
\dashline[0]{5}(215,35)(200,60)\dashline[0]{5}(185,35)(215,35)\dashline[0]{5}(200,10)(185,35)\dashline[0]{5}(215,35)(200,10)

%\drawundirectededge(AA,BB){} \drawundirectededge(BB,CC){}\drawundirectededge(CC,DD){} \drawundirectededge(DD,EE){}
%\drawundirectededge(EE,FF){}\drawundirectededge(FF,AA){}\drawundirectededge(BB,FF){}\drawundirectededge(BB,DD){}\drawundirectededge(DD,FF){}
\dashline[0]{5}(340,60)(325,35)\dashline[0]{5}(325,35)(310,10)
\dashline[0]{5}(310,10)(340,10)\dashline[0]{5}(340,10)(370,10)
\dashline[0]{5}(355,35)(370,10)\dashline[0]{5}(355,35)(340,60)\dashline[0]{5}(325,35)(355,35)
\dashline[0]{5}(325,35)(340,10) \dashline[0]{5}(340,10)(355,35)

\drawvertex(a){$\bullet$}\drawvertex(b){$\bullet$}
\drawvertex(c){$\bullet$}\drawvertex(d){$\bullet$}
\drawvertex(e){$\bullet$}\drawvertex(f){$\bullet$}
\drawvertex(A){$\bullet$}\drawvertex(B){$\bullet$}
\drawvertex(C){$\bullet$}\drawvertex(D){$\bullet$}
\drawvertex(E){$\bullet$}\drawvertex(F){$\bullet$}
\drawvertex(AA){$\bullet$}\drawvertex(BB){$\bullet$}
\drawvertex(CC){$\bullet$}\drawvertex(DD){$\bullet$}
\drawvertex(EE){$\bullet$}\drawvertex(FF){$\bullet$}
\end{picture}
\end{center}
This proves Equation \eqref{t0n}.
\end{proof}

\noindent The following lemma can be easily proven by induction,
using Equations \eqref{t1n} and \eqref{t0n}.

\begin{lem}\label{lemmafattori}
For each $n\geq 1$, $x-1$ divides $T_{1,n}(x,y)$ and $(x-1)^2$
divides $T_{0,n}(x,y)$ in $\mathbb{Z}[x,y]$.
\end{lem}
As a consequence, we can write
\begin{eqnarray}\label{semplificate}
T_{1,n}(x,y) = (x-1)N_n(x,y) \qquad \mbox{ and } \qquad T_{0,n}(x,y) =
(x-1)^2M_n(x,y),
\end{eqnarray}
with $N_n(x,y)$ and $ M_n(x,y) \in \mathbb{Z}[x,y]$.

\indent Using Equation \eqref{semplificate} for $T_{1,n}(x,y)$ and
$T_{0,n}(x,y)$, Equations \eqref{t2n}, \eqref{t1n}, \eqref{t0n}
can be rewritten as
\begin{eqnarray}\label{pigro2}
T_{2,n+1}(x,y) &=& (y-1)T_{2,n}^3
+3(x-1)T_{2,n}N_{n}^2+6T_{2,n}^2N_{n}
\end{eqnarray}
\begin{eqnarray}\label{pigro1}
N_{n+1}(x,y)&=& (y-1)T_{2,n}^2N_{n}+T_{2,n}^2M_{n} +7T_{2,n}N_{n}^2\\
&+& (x-1)\left(2T_{2,n}N_{n}M_{n}+4N_{n}^3\right)
+(x-1)^2N_{n}^2M_{n}\nonumber
\end{eqnarray}
\begin{eqnarray}\label{pigro0}
M_{n+1}(x,y)&=&(y-1)\left((x-1)N_{n}^3+3T_{2,n}N_{n}^2\right)+
12T_{2,n}N_{n}M_{n}+14N_{n}^3\\
&+&
(x-1)\left(3T_{2,n}M_{n}^2+24N_{n}^2M_{n}\right)+9(x-1)^2N_{n}M_{n}^2+(x-1)^3M_{n}^3,\nonumber
\end{eqnarray}
with initial conditions
$$
T_{2,1}(x,y)=y+2 \qquad N_1(x,y)=M_1(x,y)=1.
$$

These reduced equations turn out to be very useful for many
computations that can be done by evaluating the Tutte polynomial
in special points of the line $x=1$. Let us start by writing the
reliability polynomial $R(\Gamma_n,p)$.
\begin{prop}\label{donatellasierp}
For each $n\geq 1$, the reliability polynomial $R(\Gamma_n,p)$ is
given by
$$
R(\Gamma_n,p) =
p^{\frac{3^n+1}{2}}(1-p)^{\frac{3^n-1}{2}}T_n\left(1,\frac{1}{1-p}\right),
$$
with $T_n\left(1,\frac{1}{1-p}\right)=
T_{2,n}\left(1,\frac{1}{1-p}\right)$ and
\begin{eqnarray}\label{olga2}
T_{2,n+1}\left(1,\frac{1}{1-p}\right) =\frac{p}{1-p}T_{2,n}^3
+6T_{2,n}^2N_{n}
\end{eqnarray}
\begin{eqnarray}\label{olga1}
N_{n+1}\left(1,\frac{1}{1-p}\right)=
\frac{p}{1-p}T_{2,n}^2N_{n}+T_{2,n}^2M_{n}+7T_{2,n}N_{n}^2
\end{eqnarray}
\begin{eqnarray}\label{olga0}
M_{n+1}\left(1,\frac{1}{1-p}\right)=\frac{3p}{1-p}T_{2,n}N_{n}^2+
12T_{2,n}N_{n}M_{n}+14N_{n}^3,
\end{eqnarray}
with initial conditions
$$
T_{2,1}\left(1,\frac{1}{1-p}\right) =\frac{3-2p}{1-p} \qquad
N_1\left(1,\frac{1}{1-p}\right) = M_1\left(1,\frac{1}{1-p}\right)
=1.
$$
\end{prop}

\begin{proof}
One has $T_n\left(1,\frac{1}{1-p}\right)=
T_{2,n}\left(1,\frac{1}{1-p}\right)$, since
$T_{1,n}(1,y)=T_{0,n}(1,y)=0$, for every $y\in \mathbb{R}$ (see
Lemma \ref{lemmafattori}). Then, it suffices to apply (1) of
Theorem \ref{twopolynomials} and use Equations \eqref{pigro2},
\eqref{pigro1} and \eqref{pigro0}.
\end{proof}

\begin{os}\label{geometricremark}\rm
The analytic property $T_{1,n}(1,y)=T_{0,n}(1,y)=0$ has the
following geometric interpretation. The only nontrivial terms in
the sum \eqref{defsubgraphs} of Definition \ref{defspanning}, for
$x=1$, correspond to subgraphs $A$ such that $r(\Gamma_n)-r(A)=0$.
Since $|V(A)|=|V(\Gamma_n)|$, this means $k(A)=k(\Gamma_n)=1$ and
so $A$ must be a connected spanning subgraph of $\Gamma_n$; this
implies $A\in D_{2,n}$.
\end{os}

\begin{prop}\label{propcomplexsierp}
The complexity $\tau(\Gamma_n)$ is given by $T_n(1,1) =
T_{2,n}(1,1)$, where
\begin{eqnarray*}
T_{2,n+1}(1,1) = 6T_{2,n}^2N_{n}
\end{eqnarray*}
\begin{eqnarray*}
N_{n+1}(1,1)=T_{2,n}^2M_{n}+7T_{2,n}N_{n}^2
\end{eqnarray*}
\begin{eqnarray*}
M_{n+1}(1,1)=12T_{2,n}N_{n}M_{n}+14N_{n}^3,
\end{eqnarray*}
with initial conditions
$$
T_{2,1}(1,1)=3 \qquad N_1(1,1)=M_1(1,1)=1.
$$
\end{prop}

\begin{proof}
One can compute $T_n(1,1)$ by evaluating
$T_n\left(1,\frac{1}{1-p}\right)=T_{2,n}\left(1,\frac{1}{1-p}\right)$
in $p=0$, using Equations \eqref{olga2}, \eqref{olga1} and
\eqref{olga0}.
\end{proof}

\begin{os}\rm
These equations coincide with the relations obtained in
\cite{taiwan, noispanning, wagner1}, without using Tutte
polynomials. More precisely, one can find in \cite[Theorem
3.1]{taiwan} and \cite[Corollary 2.3]{noispanning}:
\begin{enumerate}
\item $ T_n(1,1)= \tau(\Gamma_n) =2^{\frac{3^
{n-1}-1}{2}}3^{\frac{3^n+2n-1}{4}}5^{\frac{3^{n-1}-2n+1}{4}}$;
\item $N_n(1,1)= 2^{\frac{3^{n-1}-1}{2}}3^{\frac{3^{n}-2n-1}{4}}5^{\frac{3^{n-1}+2n-3}{4}}$;
\item $M_n(1,1) =
2^{\frac{3^{n-1}-1}{2}}3^{\frac{3^{n}-6n+3}{4}}5^{\frac{3^{n-1}+6n-7}{4}}$.
\end{enumerate}
Then, the asymptotic growth constant of the spanning trees of
$\Gamma_n$ is
$$
\lim_{n\to \infty}\frac{\log(\tau(\Gamma_n))}{|V(\Gamma_n)|}=
\frac{1}{3}\log 2+\frac{1}{2}\log 3+\frac{1}{6}\log 5.
$$
\end{os}
\noindent Evaluating $T_n\left(1,\frac{1}{1-p}\right)$ in
$p=\frac{1}{2}$ gives the number of connected spanning subgraphs
of $\Gamma_n$.

\begin{prop}\label{propconnsubgrsierp}
The number of connected spanning subgraphs of $\Gamma_n$ is given
by $T_n(1,2) = T_{2,n}(1,2)$, with
\begin{eqnarray*}
T_{2,n+1}(1,2) =T_{2,n}^3 +6T_{2,n}^2N_{n}
\end{eqnarray*}
\begin{eqnarray*}
N_{n+1}(1,2)=T_{2,n}^2N_{n}+T_{2,n}^2M_{n}+7T_{2,n}N_{n}^2
\end{eqnarray*}
\begin{eqnarray*}
M_{n+1}(1,2)=3T_{2,n}N_{n}^2+ 12T_{2,n}N_{n}M_{n}+14N_{n}^3,
\end{eqnarray*}
with initial conditions
$$
T_{2,1}(1,2)=4 \qquad N_1(1,2)=M_1(1,2)=1.
$$
\end{prop}

\begin{proof}
By Lemma \ref{lemmafattori}, one has $T_{1,n}(1,y) =
T_{0,n}(1,y)=0$, for every $y\in \mathbb{R}$. Therefore
$T_n(1,2)=T_{2,n}(1,2)$ and it suffices to apply Formula (2) of
Theorem \ref{evaluations}.
\end{proof}

\begin{os}\rm
This specialization of $T_n(x,y)$ returns the equations obtained
in \cite{taiwan2} without using the Tutte polynomial.
\end{os}

Another interesting computation concerns the number of spanning
forests of the Sierpi\'{n}ski graph $\Gamma_n$.
\begin{prop}\label{spannforsierp}
The number of spanning forests of $\Gamma_n$ is given by
$$
T_{n}(2,1) = T_{2,n}(2,1) + 3N_n(2,1)+M_n(2,1),
$$
where
\begin{eqnarray*}
T_{2,n+1}(2,1)=6T_{2,n}^2N_{n}+ 3T_{2,n}N_{n}^2
\end{eqnarray*}
\begin{eqnarray*}
N_{n+1}(2,1)=T_{2,n}^2M_{n}+7T_{2,n}N_{n}^2+2T_{2,n}N_{n}M_{n}+4N_{n}^3
+N_{n}^2M_{n}
\end{eqnarray*}
\begin{eqnarray*}
M_{n+1}(2,1)=12T_{2,n}N_{n}M_{n}+3T_{2,n}M_{n}^2 +14N_{n}^3
+24N_{n}^2M_{n}+9N_{n}M_{n}^2+M_{n}^3,
\end{eqnarray*}
with initial conditions
$$
T_{2,1}(2,1)=3 \qquad N_1(2,1)=M_1(2,1)=1.
$$
\end{prop}

\begin{proof}
It suffices to apply Formula (3) of Theorem \ref{evaluations} and
to observe that $T_{1,n}(2,y)= N_n(2,y)$ and
$T_{0,n}(2,y)=M_n(2,y)$, for each $y\in \mathbb{R}$.
\end{proof}

\begin{os}\rm
This specialization of $T_n(x,y)$ returns the equations obtained
in \cite{taiwan3} without using the Tutte polynomial.
\end{os}

Next, we explicitly verify that, by evaluating the Tutte
polynomial of $\Gamma_n$ in $(2,2)$, one gets $2^{|E(\Gamma_n)|}$
(see Formula (4) of Theorem \ref{evaluations}).

\begin{prop}
For each $n\geq 1$, one has $T_n(2,2)=
2^{|E(\Gamma_{n})|}=2^{3^n}$.
\end{prop}

\begin{proof}
We prove the assertion by induction. For $n=1$, we have
$T_1(2,2)=8=2^3=2^{ |E(\Gamma_{1})|}$. Then, we recall that
$|E(\Gamma_{n})|=3^n$, therefore
$|E(\Gamma_{n+1})|=3^{n+1}=3|E(\Gamma_{n})|$. An easy computation
shows that  $T_{n+1}(2,2)= T_{n}(2,2)^3$; therefore,
$T_{n+1}(2,2)= T_{n}(2,2)^3= \left(2^{ |E(\Gamma_{n})|}\right)^3=
 2^{ 3|E(\Gamma_{n})|}= 2^{ |E(\Gamma_{n+1})|}     $.
\end{proof}

Finally, by evaluating the Tutte polynomial of $\Gamma_n$ in
$(2,0)$, we investigate the number of acyclic orientations of
$\Gamma_n$.

\begin{prop}\label{propacycsierp}
The number of acyclic orientations on $\Gamma_n$ is $T_n(2,0)$,
with
$$
T_{n+1}(2,0) = T_{n}(2,0)^3 -
2\left(T_{2,n}(2,0)+N_{n}(2,0)\right)^3
$$
and
\begin{eqnarray*}
T_{2,n+1}(2,0) = -T_{2,n}^3 +6T_{2,n}^2N_{n}+3T_{2,n}N_{n}^2
\end{eqnarray*}
\begin{eqnarray*}
N_{n+1}(2,0)=
-T_{2,n}^2N_{n}+T_{2,n}^2M_{n}+7T_{2,n}N_{n}^2+2T_{2,n}N_{n}M_{n}+4N_{n}^3
+N_{n}^2M_{n}
\end{eqnarray*}
\begin{eqnarray*}
M_{n+1}(2,0)=-3T_{2,n}N_{n}^2+ 12T_{2,n}N_{n}M_{n}+3T_{2,n}M_{n}^2
+ 13N_{n}^3+24N_{n}^2M_{n}+9N_{n}M_{n}^2+M_{n}^3,
\end{eqnarray*}
with initial conditions
$$
T_{2,1}(2,0)=2 \qquad N_1(2,0)=M_1(2,0)=1.
$$
\end{prop}

\begin{proof}
It suffices to apply Formula (5) of Theorem \ref{evaluations}.
Then, one can directly verify that $T_{2,n+1}(2,0) + 3
T_{1,n+1}(2,0) +T_{0,n+1}(2,0)$ can be rewritten as $T_{n}(2,0)^3
- 2\left(T_{2,n}(2,0)+N_{n}(2,0)\right)^3$.
\end{proof}

\begin{os}\rm
In \cite{chang}, the author obtains recursively the number $f(n)$
of acyclic orientations of $\Gamma_n$ as a sum of four
contributions, namely $f(n)=6a(n)+6b(n)+6c(n)+d(n)$. On the other
hand, via the Tutte polynomial, we need to introduce only three
contributions, corresponding to $T_{2,n}(2,0), N_n(2,0),
M_n(2,0)$. Moreover, it is not difficult to show by induction that
the following correspondences hold:
$$
T_{2,n}(2,0) = 2a(n)+b(n) \quad T_{1,n}(2,0)=a(n)+b(n)+c(n) \quad
T_{0,n}(2,0) = a(n)+2b(n)+3c(n)+d(n).
$$
\end{os}

Next, we study the chromatic polynomial of $\Gamma_n$.

\begin{prop}\label{propchromsierp}
For each $n\geq 1$, the chromatic polynomial $\chi_n(\lambda)$ of
the Sierpi\'{n}ski graph $\Gamma_n$ is
$$
\chi_n(\lambda) = (-1)^{\frac{3^n+1}{2}}\lambda P_n(\lambda),
$$
where $P_n(\lambda) = P_{2,n}(\lambda) + 3P_{1,n}(\lambda) +
P_{0,n}(\lambda)$, and
\begin{eqnarray*}
P_{2,n+1}(\lambda) =
-P_{2,n}^3-\frac{1}{\lambda}\left(6P_{2,n}^2P_{1,n}+3P_{2,n}P_{1,n}^2\right)
\end{eqnarray*}
\begin{eqnarray*}
P_{1,n+1}(\lambda)&=&
-P_{2,n}^2P_{1,n}-\frac{1}{\lambda}\left(P_{2,n}^2P_{0,n}+7P_{2,n}P_{1,n}^2\right.\\&+&\left.2P_{2,n}P_{1,n}P_{0,n}+4P_{1,n}^3
+P_{1,n}^2P_{0,n}\right)\nonumber
\end{eqnarray*}
\begin{eqnarray*}
P_{0,n+1}(\lambda)&=&-\left(3P_{2,n}P_{1,n}^2+P_{1,n}^3\right)-\frac{1}{\lambda}\left(12P_{2,n}P_{1,n}P_{0,n}\right.
\\
&+&\left.3P_{2,n}P_{0,n}^2+14P_{1,n}^3+24P_{1,n}^2P_{0,n}+9P_{1,n}P_{0,n}^2+P_{0,n}^3\right),\nonumber
\end{eqnarray*}
with initial conditions
$$
P_{2,1}(\lambda)=2 \qquad P_{1,1}(\lambda)=-\lambda \qquad P_{0,1}(\lambda) =\lambda^2.
$$
\end{prop}

\begin{proof}
It is an easy consequence of Equation (2) of Theorem
\ref{twopolynomials}, with the convention $P_{i,n}(\lambda)=
T_{i,n}(1-\lambda,0)$, for each $n\geq 1$ and $i=0,1,2$.
\end{proof}
It is known that the chromatic number $\chi(\Gamma_n)$ is $3$.
Using the Tutte polynomial, we are able to prove the following
stronger result about the colorability of $\Gamma_n$.

\begin{prop}\label{colorabilitysierp}
The graph $\Gamma_n$ is uniquely $3$-colorable.
\end{prop}

\begin{proof}
It is easy to show, by induction, that for each $n\geq 1$ one has:
$$
P_{2,n}(3)= (-1)^{n+1}2 \qquad P_{1,n}(3)= (-1)^n3 \qquad
P_{0,n}(3) = (-1)^{n+1}9.
$$
Therefore, $P_n(3) = (-1)^{n+1}2$ and so
$\chi_n(3)=(-1)^{\frac{3^n+1}{2}+n+1}6=6$, showing that $\Gamma_n$
is uniquely $3$-colorable (up to permutation of the colors).
\end{proof}

\begin{os}\rm
The same result has been proven in \cite[Theorem 3.1]{coloring},
where the author uses the stronger induction assumption that
$\Gamma_n$ is uniquely $3$-colorable and in every $3$-coloring the
outmost vertices have different colors.
\end{os}

We end this section by investigating the relationship between the
evaluation of the Tutte polynomial of the Sierpi\'{n}ski graph
$\Gamma_n$ on the hyperbola $(x-1)(y-1)=2$ and the partition
function of the Ising model on the same graph. In \cite[Theorem
3.5]{noiising}, the partition function of the Ising model on
$\Gamma_n$ has been described as
$$
Z_n= 2^{\frac{3^n+3}{2}}\cosh(\beta J)^{3^n}\Phi_n(\tanh(\beta
J)),
$$
with
$$
\Phi_n(z)=z^{\frac{3^n}{2}}\prod_{k=1}^n\phi_k^{3^{n-k}}(z)(\phi_{n+1}(z)-1),
$$
where $\phi_1(z) = \frac{z+1}{z^{1/2}}$,
$\phi_2(z)=\frac{z^2+1}{z}$ and
$\phi_k(z)=\phi_{k-1}^2(z)-3\phi_{k-1}(z)+4$, for each $k\geq 3$.

\begin{teo}\label{thmisingsierp}
For each $n\geq 1$, one has
\begin{eqnarray}\label{isingbase}
2(e^{2\beta J}-1)^{|V(\Gamma_n)|-1}e^{-\beta
J|E(\Gamma_n)|}T_n\left(\frac{e^{2\beta J}+1}{e^{2\beta
J}-1},e^{2\beta J}\right) = Z_n.
\end{eqnarray}
\end{teo}

\begin{proof}
Recall that $|E(\Gamma_n)| = 3^n$ and $|V(\Gamma_n)|=
\frac{3^n+3}{2}$. Let $e^{\beta J}=t$, so that Equation
\eqref{isingbase} can be written as
$$
\frac{2(t^2-1)^{\frac{3^n+1}{2}}}{t^{3^n}}T_n\left(\frac{t^2+1}{t^2-1},t^2\right)=
2^{\frac{3^n+3}{2}}\left(\frac{t^2+1}{2t}\right)^{3^n}\!\!\Phi_n\left(\frac{t^2-1}{t^2+1}\right),
$$
or, more explicitly,
$$
\frac{2(t^2-1)^{\frac{3^n+1}{2}}}{t^{3^n}}T_n\left(\frac{t^2+1}{t^2-1},t^2\right)=2^{\frac{3^n+3}{2}}\left(\frac{t^2+1}{2t}\right)^{3^n}\!\!\!
\left(\frac{t^2-1}{t^2+1}\right)^{\frac{3^n}{2}}\!
\prod_{k=1}^n\phi_k^{3^{n-k}}(z)(\phi_{n+1}(z)-1)\left|_{z=\frac{t^2-1}{t^2+1}}\right..
$$
In order to prove this equation we put, for each $n\geq 1$:
$$
A_n(x,y)= T_{2,n}(x,y)+T_{1,n}(x,y) \qquad
B_n(x,y)=2T_{1,n}(x,y)+T_{0,n}(x,y)
$$
and
$$
C_n = \frac{2(t^2-1)^{\frac{3^n+1}{2}}}{t^{3^n}} \qquad \qquad
D_n= 2^{\frac{3^n+3}{2}}\left(\frac{t^2+1}{2t}\right)^{3^n}\!\!
{\left(\frac{t^2-1}{t^2+1}\right)}^{\frac{3^n}{2}}\!\!\!\!\!.
$$
Observe that
$$
T_n(x,y)=A_n(x,y)+B_n(x,y)
$$
and
$$
C_{n+1}= \frac{C_n^3}{4(t^2-1)} \qquad \qquad D_{n+1}=
\frac{D_n^3}{8}.
$$
Therefore, to prove the required equation is equivalent to prove
the following equalities:
\begin{eqnarray*}
C_nA_n\left(\frac{t^2+1}{t^2-1}, t^2\right) =
D_n\prod_{k=1}^n\phi_k^{3^{n-k}}(z)\left|_{z=\frac{t^2-1}{t^2+1}}\right.
\end{eqnarray*}
\begin{eqnarray*}
C_nB_n\left(\frac{t^2+1}{t^2-1}, t^2\right) =D_n
\prod_{k=1}^n\phi_k^{3^{n-k}}(z)(\phi_{n+1}(z)-2)\left|_{z=\frac{t^2-1}{t^2+1}}\right..
\end{eqnarray*}
We can prove them by induction, observing that Equations
\eqref{t2n}, \eqref{t1n} and \eqref{t0n}, evaluated on the
hyperbola $(x-1)(y-1)=2$, give:
$$
A_{n+1}= \frac{1}{2}(y-1)A_n^2(2A_n+B_n)
$$
$$
B_{n+1} = \frac{1}{2}(y-1)B_n(2A_n+B_n)(A_n+B_n),
$$
with initial conditions
$$
A_1\left(\frac{y+1}{y-1},y\right)=\frac{y(y+1)}{y-1} \qquad
B_1\left(\frac{y+1}{y-1},y\right)= \frac{4y}{(y-1)^2}.
$$
Indeed, for $n=1$, one has
$$
C_1A_1\left(\frac{t^2+1}{t^2-1}, t^2\right) =
D_1\phi_1(z)\left|_{z=\frac{t^2-1}{t^2+1}}\right. =
\frac{2(t^4-1)}{t}
$$
and
$$
C_1B_1\left(\frac{t^2+1}{t^2-1}, t^2\right) =
D_1\phi_1(z)(\phi_2(z)-2)\left|_{z=\frac{t^2-1}{t^2+1}}\right.=\frac{8}{t};
$$
thus, the assertion is true. Now
\begin{eqnarray*}
C_{n+1}A_{n+1}&=& \frac{1}{8}C_n^2A_n^2(2C_nA_n+C_nB_n)\\
&=&
\frac{1}{8}D_n^3\left(\prod_{k=1}^n\phi_k^{3^{n-k}}(z)\right)^3
\left(2+(\phi_{n+1}(z)-2)\right)\\
&=& D_{n+1}\prod_{k=1}^{n+1}\phi_k^{3^{n+1-k}}(z).
\end{eqnarray*}
Similarly, one has
\begin{eqnarray*}
C_{n+1}B_{n+1}&=& \frac{1}{8}C_nB_n(2C_nA_n+C_nB_n)(C_nA_n+C_nB_n)\\
&=&
\frac{1}{8}D_n^3\left(\prod_{k=1}^n\phi_k^{3^{n-k}}(z)\right)^3
(\phi_{n+1}(z)-2)(2+\phi_{n+1}(z)-2)(1+\phi_{n+1}(z)-2)\\
&=&
D_{n+1}\prod_{k=1}^{n+1}\phi_k^{3^{n+1-k}}(z)\left(\phi_{n+1}^2(z)-3\phi_{n+1}(z)+2\right)\\
&=&
D_{n+1}\prod_{k=1}^{n+1}\phi_k^{3^{n+1-k}}(z)(\phi_{n+2}(z)-2).
\end{eqnarray*}
\end{proof}

%%%%%%%%%%%%%%%%%%%%%%%%%%%%%%%%%%%%%%%%%%%%%%%%%%%%%%%%%%%%%%%%%%%%%%%%%%%%%%%%%%%%%%%%%%%%%%%%%%%%%%%%%%%%%%%%%%%%%%%%%%
\section{The Tutte polynomial of the Schreier graphs of the Hanoi Towers group}\label{sezione hanoi}

In this section, we study the Tutte polynomial of the Schreier
graphs $\{\Sigma_n\}_{n\geq 1}$ of the Hanoi Towers group
$H^{(3)}$. The strategy is still to use the self-similarity of the
graphs in order to approach recursively the problem.

\subsection{The Schreier graphs of the Hanoi Towers group}\label{sectionhanoi}

The Hanoi Towers groups $H^{(3)}$ is generated by the
automorphisms of the ternary rooted tree having the following
self-similar form:
$$
a= (01)(id,id,a) \qquad  b=(02)(id,b,id) \qquad c=(12)(c,id,id),
$$
where $(01), (02)$ and $(12)$ are elements of the symmetric group
$Sym(3)$ acting on the set $X=\{0,1,2\}$. Observe that $a,b,c$ are
involutions. The associated Schreier graphs are self-similar in
the sense of \cite{wagner2}, that is, $\Sigma_{n}$ contains three
copies of $\Sigma_{n-1}$  glued together by three edges, that we
call \emph{special edges}. Their endpoints will be called
\emph{special vertices} of $\Sigma_n$. These graphs can be
recursively constructed via the following substitutional rules
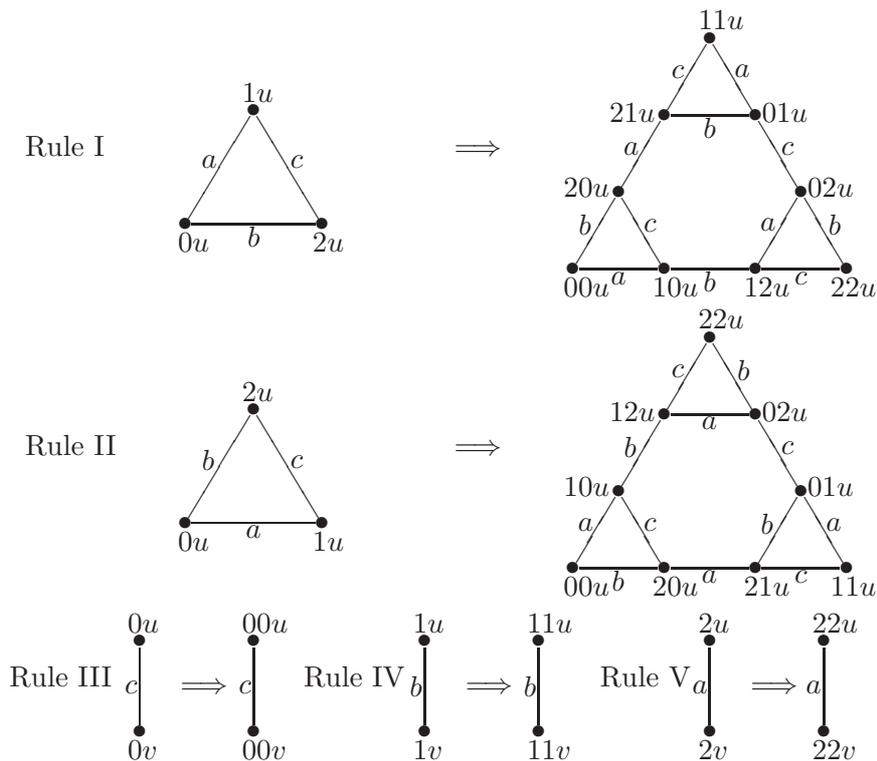
\cite{hanoi},

\begin{center}
\begin{picture}(400,115)
\letvertex A=(240,10)\letvertex B=(260,44)
\letvertex C=(280,78)\letvertex D=(300,112)
\letvertex E=(320,78)\letvertex F=(340,44)
\letvertex G=(360,10)\letvertex H=(320,10)\letvertex I=(280,10)

\letvertex L=(70,30)\letvertex M=(130,30)
\letvertex N=(100,80)

\put(236,-2){$00u$}\put(236,42){$20u$}\put(256,75){$21u$}
\put(295,116){$11u$}\put(323,75){$01u$}\put(343,42){$02u$}\put(353,-2){$22u$}
\put(315,-2){$12u$}\put(275,-2){$10u$}

\put(67,18){$0u$}\put(126,18){$2u$}\put(95,84){$1u$}\put(188,60){$\Longrightarrow$}
\put(0,60){Rule I}

\drawvertex(A){$\bullet$}\drawvertex(B){$\bullet$}
\drawvertex(C){$\bullet$}\drawvertex(D){$\bullet$}
\drawvertex(E){$\bullet$}\drawvertex(F){$\bullet$}
\drawvertex(G){$\bullet$}\drawvertex(H){$\bullet$}
\drawvertex(I){$\bullet$}
\drawundirectededge(A,B){$b$}\drawundirectededge(B,C){$a$}\drawundirectededge(C,D){$c$}
\drawundirectededge(D,E){$a$}\drawundirectededge(E,C){$b$}\drawundirectededge(E,F){$c$}\drawundirectededge(F,G){$b$}
\drawundirectededge(B,I){$c$}\drawundirectededge(H,F){$a$}\drawundirectededge(H,I){$b$}
\drawundirectededge(I,A){$a$}\drawundirectededge(G,H){$c$}

\drawvertex(L){$\bullet$}
\drawvertex(M){$\bullet$}\drawvertex(N){$\bullet$}
\drawundirectededge(M,L){$b$}\drawundirectededge(N,M){$c$}\drawundirectededge(L,N){$a$}

\end{picture}
\end{center}

\begin{center}
\begin{picture}(400,120)
\letvertex A=(240,10)\letvertex B=(260,44)
\letvertex C=(280,78)\letvertex D=(300,112)
\letvertex E=(320,78)\letvertex F=(340,44)
\letvertex G=(360,10)\letvertex H=(320,10)\letvertex I=(280,10)

\letvertex L=(70,30)\letvertex M=(130,30)
\letvertex N=(100,80)

\put(236,-2){$00u$}\put(236,42){$10u$}\put(256,75){$12u$}
\put(295,116){$22u$}\put(323,75){$02u$}\put(343,42){$01u$}\put(353,-2){$11u$}
\put(315,-2){$21u$}\put(275,-2){$20u$}

\put(67,18){$0u$}\put(126,18){$1u$}\put(95,84){$2u$}\put(188,60){$\Longrightarrow$}
\put(0,60){Rule II}
\drawvertex(A){$\bullet$}\drawvertex(B){$\bullet$}
\drawvertex(C){$\bullet$}\drawvertex(D){$\bullet$}
\drawvertex(E){$\bullet$}\drawvertex(F){$\bullet$}
\drawvertex(G){$\bullet$}\drawvertex(H){$\bullet$}
\drawvertex(I){$\bullet$}
\drawundirectededge(A,B){$a$}\drawundirectededge(B,C){$b$}\drawundirectededge(C,D){$c$}
\drawundirectededge(D,E){$b$}\drawundirectededge(E,C){$a$}\drawundirectededge(E,F){$c$}\drawundirectededge(F,G){$a$}
\drawundirectededge(B,I){$c$}\drawundirectededge(H,F){$b$}\drawundirectededge(H,I){$a$}
\drawundirectededge(I,A){$b$}\drawundirectededge(G,H){$c$}

\drawvertex(L){$\bullet$}
\drawvertex(M){$\bullet$}\drawvertex(N){$\bullet$}
\drawundirectededge(M,L){$a$}\drawundirectededge(N,M){$c$}\drawundirectededge(L,N){$b$}
\end{picture}
\end{center}
\begin{center}
\begin{picture}(400,60)
\letvertex A=(50,10)\letvertex B=(100,10)
\letvertex C=(175,10)\letvertex D=(225,10)
\letvertex E=(300,10)\letvertex F=(350,10)
\letvertex G=(50,50)\letvertex H=(100,50)
\letvertex I=(175,50)\letvertex L=(225,50)
\letvertex M=(300,50)\letvertex N=(350,50)

\put(45,54){$0u$}\put(45,-2){$0v$}
\put(95,-2){$00v$}\put(95,54){$00u$}\put(170,-2){$1v$}\put(170,54){$1u$}\put(220,-2){$11v$}
\put(220,54){$11u$}\put(295,-2){$2v$}\put(295,54){$2u$}\put(345,-2){$22v$}\put(345,54){$22u$}

\put(68,27){$\Longrightarrow$}\put(193,27){$\Longrightarrow$}\put(318,27){$\Longrightarrow$}
\put(-7,30){Rule III} \put(122,30){Rule IV} \put(252,30){Rule V}

\drawvertex(A){$\bullet$}\drawvertex(B){$\bullet$}
\drawvertex(C){$\bullet$}\drawvertex(D){$\bullet$}
\drawvertex(E){$\bullet$}\drawvertex(F){$\bullet$}
\drawvertex(G){$\bullet$}\drawvertex(H){$\bullet$}
\drawvertex(I){$\bullet$}\drawvertex(L){$\bullet$}
\drawvertex(M){$\bullet$}\drawvertex(N){$\bullet$}

\drawundirectededge(A,G){$c$}\drawundirectededge(B,H){$c$}\drawundirectededge(C,I){$b$}
\drawundirectededge(D,L){$b$}\drawundirectededge(E,M){$a$}\drawundirectededge(F,N){$a$}
\end{picture}
\end{center}
where the word $u$ in Rule I and Rule II can also be the empty
word and the words $u$ and $v$ in Rules III, IV, V can also
satisfy $u=v$ (in this case we get the three loops of $\Sigma_n$).
The starting point is the Schreier graph $\Sigma_1$ of the first
level. We also draw a picture of $\Sigma_2$.
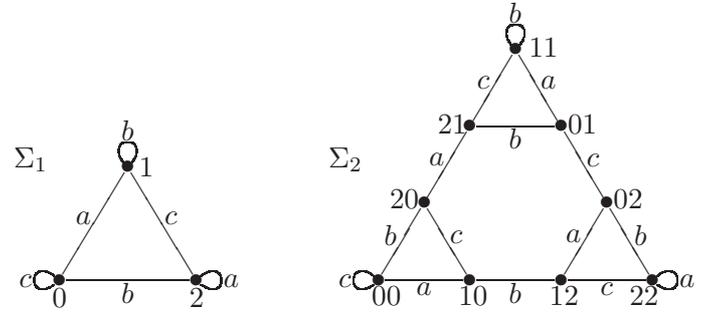
\begin{center}
\begin{picture}(400,125)

\letvertex L=(60,10)\letvertex M=(120,10)
\letvertex N=(90,60)

\put(57,-2){$0$}\put(117,-2){$2$}\put(95,56){$1$}\put(40,60){$\Sigma_1$}

\drawvertex(L){$\bullet$}
\drawvertex(M){$\bullet$}\drawvertex(N){$\bullet$}
\drawundirectededge(M,L){$b$}\drawundirectededge(N,M){$c$}\drawundirectededge(L,N){$a$}

\drawundirectedloop[r](M){$a$}\drawundirectedloop(N){$b$}\drawundirectedloop[l](L){$c$}

\letvertex A=(200,10)\letvertex B=(220,44)
\letvertex C=(240,78)

\letvertex D=(260,112)
\letvertex E=(280,78)\letvertex F=(300,44)
\letvertex G=(320,10)\letvertex H=(280,10)\letvertex I=(240,10)

\put(197,-1){$00$}\put(205,42){$20$}\put(226,75){$21$}
\put(266,109){$11$}\put(283,75){$01$}\put(303,42){$02$}\put(310,-1){$22$}
\put(275,-1){$12$}\put(235,-1){$10$}\put(178,60){$\Sigma_2$}

\drawvertex(A){$\bullet$}\drawvertex(B){$\bullet$}
\drawvertex(C){$\bullet$} \drawvertex(D){$\bullet$}
\drawvertex(E){$\bullet$}\drawvertex(F){$\bullet$}
\drawvertex(G){$\bullet$}\drawvertex(H){$\bullet$}
\drawvertex(I){$\bullet$}
\drawundirectededge(A,B){$b$}\drawundirectededge(B,C){$a$}\drawundirectededge(C,D){$c$}
\drawundirectededge(D,E){$a$}\drawundirectededge(E,C){$b$}\drawundirectededge(E,F){$c$}\drawundirectededge(F,G){$b$}
\drawundirectededge(B,I){$c$}\drawundirectededge(H,F){$a$}\drawundirectededge(H,I){$b$}
\drawundirectededge(I,A){$a$}\drawundirectededge(G,H){$c$}

\drawundirectedloop[l](A){$c$}\drawundirectedloop(D){$b$}\drawundirectedloop[r](G){$a$}

\end{picture}
\end{center}

\begin{os}\rm
Observe that, for each $n\geq 1$, the graph $\Sigma_n$ has three
loops, centered at the outmost vertices $0^n,1^n$ and $2^n$,
labelled by $c,b$ and $a$, respectively. This is an easy
consequence of the definition of the generators $a,b$ and $c$ of
$H^{(3)}$. Moreover, these are the only loops in $\Sigma_n$.
\end{os}
\begin{os}\rm
As we already mentioned in Section \ref{sezione serpin}, the
graphs $\{\Sigma_n\}_{n\geq 1}$ are very close to the
Sierpi\'{n}ski graphs $\{\Gamma_n\}_{n\geq 1}$. Indeed, each
$\Gamma_n$ can be obtained from $\Sigma_n$ by removing loops and
contracting all the special edges of $\Sigma_n$ at each step.
\end{os}

\subsection{The Tutte polynomial of $\Sigma_n$}\label{4.2}

In this section we study the Tutte polynomial of the graph
$\Sigma_n$, considered without its loops. However, the presence of
the three loops would change the polynomial only by the factor
$y^3$ (see Definition \ref{contracting}). Moreover, for our
purposes, we can forget the generator labelling of the edges of
$\Sigma_n$, as well as the word labelling of its vertices, and we
can regard it as an unlabelled graph. We represent the
self-similar structure of $\Sigma_n$ by the following picture.
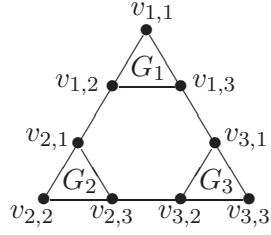
\begin{center}
\begin{picture}(400,85)
\letvertex A=(200,75)
\letvertex B=(185,50) \letvertex F=(215,50)
\letvertex C=(170,25)  \letvertex E=(230,25)
\letvertex D=(155,0)  \letvertex G=(245,0)
\letvertex N=(185,0)  \letvertex L=(215,0)

\put(193,82){$v_{1,1}$}\put(160,50){$v_{1,2}$}\put(220,50){$v_{1,3}$}
\put(147,26){$v_{2,1}$}\put(234,26){$v_{3,1}$}\put(140,-8){$v_{2,2}$}
\put(175,-8){$v_{2,3}$}\put(205,-8){$v_{3,2}$}\put(235,-8){$v_{3,3}$}

%\put(177,-13){$P_3$}\put(210,-13){$P_4$}\put(235,22){$P_5$}\put(220,47){$P_6$}\put(168,47){$P_1$}\put(152,22){$P_2$}

\put(193,55){$G_1$}\put(163,5){$G_2$}\put(223,5){$G_3$}

\drawundirectededge(D,G){}
\drawundirectededge(A,D){} \drawundirectededge(A,G){}
   \drawundirectededge(L,E){}
\drawundirectededge(B,F){} \drawundirectededge(C,N){}

\drawvertex(A){$\bullet$}\drawvertex(B){$\bullet$}
\drawvertex(C){$\bullet$}
 \drawvertex(D){$\bullet$}
\drawvertex(E){$\bullet$}\drawvertex(F){$\bullet$}
\drawvertex(G){$\bullet$}
\drawvertex(N){$\bullet$}\drawvertex(L){$\bullet$}
\end{picture}
\end{center}
\vspace{0.5cm} More precisely, the graph $\Sigma_n$ is the union
of three copies $G_1, G_2$ and $G_3$ of $\Sigma_{n-1}$, joint by
the special edges. For each $i=1,2,3$, we denote by $v_{i,j}$,
with $j=1,2,3$, the upmost, the leftmost and the rightmost vertex
of $G_i$, respectively. Moreover, it is not difficult to prove by
induction the following equalities:
$$
|V(\Sigma_n)| =  3^n \qquad \qquad
|E(\Sigma_n)|=\frac{3^{n+1}-3}{2}.
$$
Note that for $n=1$, $\Sigma_1=\Gamma_1$, and so everything is
already known. As for the Sierpi\'{n}ski graphs, we introduce the
following partition of the set of the spanning subgraphs of
$\Sigma_n$:
\begin{itemize}
\item $F_{2,n}$ denotes the set of spanning subgraphs of  $\Sigma_n$,
where the three outmost vertices belong to the same connected
component;
\item $F_{1,n}^u$ denotes the set of spanning subgraphs of $\Sigma_n$, where the leftmost and rightmost
vertices belong to the same connected component, and the upmost
one belongs to a different connected component. Similarly, by
rotation, $F_{1,n}^r$ (respectively $F_{1,n}^l$) denotes the set
of spanning subgraphs of $\Sigma_n$, where the rightmost
(respectively leftmost) vertex is not in the same connected
component containing the two other outmost vertices;
\item $F_{0,n}$ denotes the set of spanning subgraphs of $\Sigma_n$, where the three outmost
vertices belong to three different connected components.
\end{itemize}
As in Section \ref{sezione serpin}, to draw a subgraph of
$\Sigma_n$ of the previous types, we will use the following
notation.
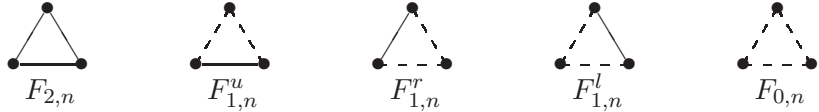
\begin{center}
\begin{picture}(400,40)
\letvertex A=(25,35)\letvertex B=(10,10)
\letvertex C=(40,10)

\letvertex a=(105,35)\letvertex b=(90,10)
\letvertex c=(120,10)

\letvertex aa=(185,35)\letvertex bb=(170,10)
\letvertex cc=(200,10)

\letvertex aaa=(265,35)\letvertex bbb=(250,10)
\letvertex ccc=(280,10)

\letvertex AAA=(345,35)\letvertex BBB=(330,10)
\letvertex CCC=(360,10)

\dashline[0]{4}(90,10)(105,35)
\dashline[0]{4}(120,10)(105,35)

\dashline[0]{4}(200,10)(185,35) \dashline[0]{4}(200,10)(170,10)

\dashline[0]{4}(250,10)(265,35) \dashline[0]{4}(250,10)(280,10)

\dashline[0]{4}(330,10)(345,35)
\dashline[0]{4}(360,10)(345,35)
\dashline[0]{4}(330,10)(360,10)

\put(15,-5){$F_{2,n}$}\put(95,-5){$F_{1,n}^u$}\put(175,-5)
{$F_{1,n}^r$}\put(255,-5){$F_{1,n}^l$}\put(335,-5){$F_{0,n}$}

\drawundirectededge(A,B){} \drawundirectededge(B,C){}
\drawundirectededge(C,A){}

\drawundirectededge(aa,bb){}

\drawundirectededge(aaa,ccc){}

%\drawundirectededge(a,b){}\drawundirectededge(c,a){}

\drawundirectededge(b,c){}

%\drawundirectededge(AA,BB){} \drawundirectededge(BB,CC){}
%\drawundirectededge(CC,AA){}

\drawvertex(a){$\bullet$}\drawvertex(b){$\bullet$}
\drawvertex(c){$\bullet$}
\drawvertex(A){$\bullet$}\drawvertex(B){$\bullet$}
\drawvertex(C){$\bullet$}
\drawvertex(AAA){$\bullet$}\drawvertex(BBB){$\bullet$}
\drawvertex(CCC){$\bullet$}

\drawvertex(aa){$\bullet$}\drawvertex(bb){$\bullet$}
\drawvertex(cc){$\bullet$}
\drawvertex(aaa){$\bullet$}\drawvertex(bbb){$\bullet$}
\drawvertex(ccc){$\bullet$}
\end{picture}
\end{center}
Observe that
$$
F_{2,n} \sqcup F_{1,n}^u\sqcup
F_{1,n}^r\sqcup F_{1,n}^l\sqcup F_{0,n}
$$
is a partition of the set of spanning subgraphs of $\Sigma_n$, for
each $n\geq 1$. Next, let us simply denote by $H_n(x,y)$ the Tutte
polynomial $T(\Sigma_n;x,y)$ of $\Sigma_n$ and define, for every
$n\geq 1$, the following polynomials:
\begin{itemize}
\item $\displaystyle H_{2,n}(x,y)= \sum_{A\in F_{2,n}}(x-1)^{r(\Sigma_n)-r(A)}(y-1)^{n(A)}$;
\item $\displaystyle H_{1,n}^u(x,y)= \sum_{A\in F_{1,n}^u}(x-1)^{r(\Sigma_n)-r(A)}(y-1)^{n(A)}$;
%\item $\displaystyle H_{1,n}^r(x,y)= \sum_{A\in F_{1,n}^r}(x-1)^{r(\Sigma_n)-r(A)}(y-1)^{n(A)}$;
%\item $\displaystyle H_{1,n}^l(x,y)= \sum_{A\in F_{1,n}^l}(x-1)^{r(\Sigma_n)-r(A)}(y-1)^{n(A)}$;
\item $\displaystyle H_{0,n}(x,y)= \sum_{A\in F_{0,n}}(x-1)^{r(\Sigma_n)-r(A)}(y-1)^{n(A)}$.
\end{itemize}
Similarly, we define $H_{1,n}^r(x,y)$ and $H_{1,n}^l(x,y)$, by
taking sums over $F_{1,n}^r$ and $F_{1,n}^l$, respectively. Note
that, by the rotational-invariance of the graph $\Sigma_n$, one
has
$$
H_{1,n}^u(x,y) =  H_{1,n}^r(x,y) = H_{1,n}^l(x,y),
$$
so that we can simply use the notation $H_{1,n}(x,y)$ to denote
one of these three polynomials. According with Definition
\ref{defspanning} of the Tutte polynomial, we have:
$$
H_n(x,y) = H_{2,n}(x,y) + 3H_{1,n}(x,y) + H_{0,n}(x,y).
$$
Also in this case, we give a recursive formula for $H_n(x,y)$,
providing recursive formulas for $H_{2,n}(x,y), H_{1,n}(x,y)$ and
$H_{0,n}(x,y)$ (Theorem \ref{noname}). The main difference with
respect to the case of the Sierpi\'{n}ski graphs is that, now, a
spanning subgraph $A$ of $\Sigma_{n+1}$ is not determined by its
restrictions $A_1,A_2$ and $A_3$ to the three copies $G_1, G_2$
and $G_3$ of $\Sigma_{n}$. In fact, the three special edges do not
belong to any of the copies of the $\Sigma_{n}$. Therefore, in
this case, we need to specify how many special edges belong to the
subgraph $A$ of $\Sigma_{n+1}$. Once we fix them, then we have the
same correspondence as before, i.e., a spanning subgraph $A$ in
$\Sigma_{n+1}$ is determined by the special edges that it contains
and by its restrictions to the three copies $G_1, G_2$ and $G_3$
of $\Sigma_{n}$.

Therefore, Equation \eqref{defsubgraphs} of Definition
\ref{defspanning} can be rewritten as
$$
H_{n+1}(x,y) = \sum_{A_i\subseteq G_i, i=1,2,3
}(x-1)^{r(\Sigma_{n+1})-r(A)}(y-1)^{n(A)}.
$$
Firstly, observe that $r(\Sigma_{n+1})=3r(\Sigma_{n})+2$ and
$|V(A)|=|V(A_1)|+|V(A_2)|+|V(A_3)|$, for every spanning subgraph
$A$ of $\Sigma_{n+1}$. Next, we have to understand how $r(A)$ and
$n(A)$ depend on $r(A_i)$ and $n(A_i)$, for $i=1,2,3$. Note that
the number of special edges belonging to $A$ plays a crucial role.
Moreover, we still have to separately consider the case in which
the special vertices belongs to the same connected components:
this can only happen when all the special edges belong to $A$.

\emph{Case I: All the special edges are in $A$}.

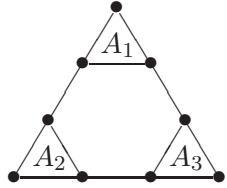
\begin{center}
\begin{picture}(400,80)
\letvertex A=(200,75)
\letvertex B=(185,50) \letvertex F=(215,50)
\letvertex C=(170,25)  \letvertex E=(230,25)
\letvertex D=(155,0)  \letvertex G=(245,0)
\letvertex N=(185,0)  \letvertex L=(215,0)

\put(193,55){$A_1$}\put(163,5){$A_2$}\put(223,5){$A_3$}

\drawundirectededge(D,G){}
\drawundirectededge(A,D){} \drawundirectededge(A,G){}
   \drawundirectededge(L,E){}
\drawundirectededge(B,F){} \drawundirectededge(C,N){}

\drawvertex(A){$\bullet$}\drawvertex(B){$\bullet$}
\drawvertex(C){$\bullet$}
 \drawvertex(D){$\bullet$}
\drawvertex(E){$\bullet$}\drawvertex(F){$\bullet$}
\drawvertex(G){$\bullet$}
\drawvertex(N){$\bullet$}\drawvertex(L){$\bullet$}
\end{picture}
\end{center}

This case is analogous to the case of the Sierpi\'{n}ski graphs.
If in the spanning subgraph $A$, obtained by the union of the
special edges and $A_1,A_2$ and $A_3$, the special vertices are in
the same connected component and, for each $i=1,2,3$, the vertices
$v_{i,j}$ and $v_{i,k}$ are in the same connected component of
$A_i$, for $j,k\neq i$, then  it is easy to check that
$$
k(A) = k(A_1) +k(A_2) + k(A_3) -2 \qquad \mbox{and} \qquad r(A) =
r(A_1) + r(A_2) + r(A_3) +2.
$$
Moreover, one has
\begin{eqnarray*}
n(A)&=&(|E(A_1)|+|E(A_2)|+|E(A_3)|+3)-(|V(A_1)|+|V(A_2)|+|V(A_3)|)\\&+&
(k(A_1) +k(A_2) + k(A_3) -2)\\
&=& n(A_1) + n(A_2) + n(A_3) + 1.
\end{eqnarray*}
Hence, for such a spanning subgraph $A$ of $\Sigma_{n+1}$ (of
\lq\lq first type\rq\rq), one gets:
$$
r(\Sigma_{n+1})-r(A) = \sum_{i=1}^3 (r(\Sigma_{n})-r(A_i)) \qquad
\mbox { and } \qquad n(A)=n(A_1) + n(A_2) + n(A_3) + 1
$$
and so
$$
(x-1)^{r(\Sigma_{n+1})-r(A)}(y-1)^{n(A)}=(y-1) \prod_{i=1}^3
(x-1)^{r(\Sigma_{n})-r(A_i)}(y-1)^{n(A_i)}.
$$
If in the spanning subgraph $A$, obtained by the union of the
special edges with $A_1,A_2$ and $A_3$, the vertices $v_{i,j}$ and
$v_{i,k}$ in $A_i$ do not belong to the same connected component
of $A_i$, for some $i=1,2,3$ and $j,k\neq i$, then it is easy to
verify that
$$
k(A) = k(A_1) +k(A_2) + k(A_3) -3 \qquad \mbox{and} \qquad r(A) =
r(A_1) + r(A_2) + r(A_3)+3.
$$
Therefore, one has in this case
\begin{eqnarray*}
n(A)
&=&(|E(A_1)|+|E(A_2)|+|E(A_3)|+3)-(|V(A_1)|+|V(A_2)|+|V(A_3)|)\\&+&
(k(A_1) +k(A_2) + k(A_3) -3)\\ &=& n(A_1) + n(A_2) + n(A_3).
\end{eqnarray*}
Hence, for such a spanning subgraph $A$ of $\Sigma_{n+1}$ (of
\lq\lq second type\rq\rq), one gets:
$$
r(\Sigma_{n+1})-r(A) = \sum_{i=1}^3 (r(\Sigma_{n})-r(A_i))-1
\qquad \mbox { and } \qquad n(A)=n(A_1) + n(A_2) + n(A_3)
$$
and so
$$
(x-1)^{r(\Sigma_{n+1})-r(A)}(y-1)^{n(A)}=   \frac{1}{(x-1)}
\prod_{i=1}^3   (x-1)^{r(\Sigma_{n})-r(A_i)}(y-1)^{n(A_i)}.
$$

\emph{Case II: Only two  special edges are in $A$}.

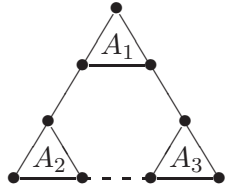
\begin{center}
\begin{picture}(400,80)
\letvertex A=(200,75)
\letvertex B=(185,50) \letvertex F=(215,50)
\letvertex C=(170,25)  \letvertex E=(230,25)
\letvertex D=(155,0)  \letvertex G=(245,0)
\letvertex N=(185,0)  \letvertex L=(215,0)

\put(193,55){$A_1$}\put(163,5){$A_2$}\put(223,5){$A_3$}

\drawundirectededge(D,N){} \drawundirectededge(L,G){}
\drawundirectededge(A,D){} \drawundirectededge(A,G){}
   \drawundirectededge(L,E){}
\drawundirectededge(B,F){} \drawundirectededge(C,N){}

\drawvertex(A){$\bullet$}\drawvertex(B){$\bullet$}
\drawvertex(C){$\bullet$}
 \drawvertex(D){$\bullet$}
\drawvertex(E){$\bullet$}\drawvertex(F){$\bullet$}
\drawvertex(G){$\bullet$}
\drawvertex(N){$\bullet$}\drawvertex(L){$\bullet$}

\dashline[0]{4}(185,0) (215,0)
\end{picture}
\end{center}

In this case
$$
k(A) = k(A_1) +k(A_2) + k(A_3) -2 \qquad \mbox{and} \qquad r(A) =
r(A_1) + r(A_2) + r(A_3)+2.
$$
One has in this case
\begin{eqnarray*}
n(A)
&=&(|E(A_1)|+|E(A_2)|+|E(A_3)|+2)-(|V(A_1)|+|V(A_2)|+|V(A_3)|)\\&+&
(k(A_1) +k(A_2) + k(A_3) -2)\\ &=& n(A_1) + n(A_2) + n(A_3).
\end{eqnarray*}
Hence, for such a spanning subgraph $A$ of $\Sigma_{n+1}$ (of
\lq\lq third type\rq\rq), one gets:
$$
r(\Sigma_{n+1})-r(A) = \sum_{i=1}^3 (r(\Sigma_{n})-r(A_i)) \qquad
\mbox { and } \qquad n(A)=n(A_1) + n(A_2) + n(A_3)
$$
and so
$$
(x-1)^{r(\Sigma_{n+1})-r(A)}(y-1)^{n(A)}= \prod_{i=1}^3
(x-1)^{r(\Sigma_{n})-r(A_i)}(y-1)^{n(A_i)}.
$$
\newpage
\emph{Case III: Only one special edge is in $A$}.

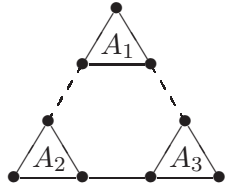
\begin{center}
\begin{picture}(400,80)
\letvertex A=(200,75)
\letvertex B=(185,50) \letvertex F=(215,50)
\letvertex C=(170,25)  \letvertex E=(230,25)
\letvertex D=(155,0)  \letvertex G=(245,0)
\letvertex N=(185,0)  \letvertex L=(215,0)
\put(193,55){$A_1$}\put(163,5){$A_2$}\put(223,5){$A_3$}

\dashline[0]{4}(185,50)(170,25)

\dashline[0]{4}(215,50)(230,25)

\drawundirectededge(A,B){} \drawundirectededge(C,D){}
\drawundirectededge(D,N){} \drawundirectededge(L,G){}
\drawundirectededge(A,F){} \drawundirectededge(L,E){}
\drawundirectededge(B,F){} \drawundirectededge(C,N){}
\drawundirectededge(L,N){}\drawundirectededge(E,G){}
\drawvertex(A){$\bullet$}\drawvertex(B){$\bullet$}
\drawvertex(C){$\bullet$} \drawvertex(D){$\bullet$}
\drawvertex(E){$\bullet$}\drawvertex(F){$\bullet$}
\drawvertex(G){$\bullet$}
\drawvertex(N){$\bullet$}\drawvertex(L){$\bullet$}
\end{picture}
\end{center}

In this case
$$
k(A) = k(A_1) +k(A_2) + k(A_3) -1\qquad \mbox{and} \qquad r(A) =
r(A_1) + r(A_2) + r(A_3)+1.
$$
Moreover, one has in this case
\begin{eqnarray*}
n(A)
&=&(|E(A_1)|+|E(A_2)|+|E(A_3)|+1)-(|V(A_1)|+|V(A_2)|+|V(A_3)|)\\&+&
(k(A_1) +k(A_2) + k(A_3) -1)\\ &=& n(A_1) + n(A_2) + n(A_3).
\end{eqnarray*}
Hence, for such a spanning subgraph $A$ of $\Sigma_{n+1}$ (of
\lq\lq fourth type\rq\rq), one gets:
$$
r(\Sigma_{n+1})-r(A) = \sum_{i=1}^3 (r(\Sigma_{n})-r(A_i))+1
\qquad \mbox { and } \qquad n(A)=n(A_1) + n(A_2) + n(A_3)
$$
and so
$$
(x-1)^{r(\Sigma_{n+1})-r(A)}(y-1)^{n(A)}= (x-1) \prod_{i=1}^3
(x-1)^{r(\Sigma_{n})-r(A_i)}(y-1)^{n(A_i)}.
$$
{\emph{Case IV: No special edge is in $A$}}.

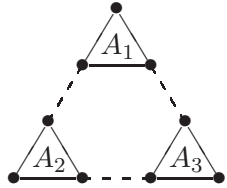
\begin{center}
\begin{picture}(400,80)
\letvertex A=(200,75)
\letvertex B=(185,50) \letvertex F=(215,50)
\letvertex C=(170,25)  \letvertex E=(230,25)
\letvertex D=(155,0)  \letvertex G=(245,0)
\letvertex N=(185,0)  \letvertex L=(215,0)

\dashline[0]{4}(185,50)(170,25)

\dashline[0]{4}(215,50)(230,25)

\dashline[0]{4}(185,0)(215,0)

\put(193,55){$A_1$}\put(163,5){$A_2$}\put(223,5){$A_3$}

\drawundirectededge(A,B){} \drawundirectededge(C,D){}
\drawundirectededge(D,N){} \drawundirectededge(L,G){}
\drawundirectededge(A,F){}\drawundirectededge(E,G){}

   \drawundirectededge(L,E){}
\drawundirectededge(B,F){} \drawundirectededge(C,N){}

\drawvertex(A){$\bullet$}\drawvertex(B){$\bullet$}
\drawvertex(C){$\bullet$}
 \drawvertex(D){$\bullet$}
\drawvertex(E){$\bullet$}\drawvertex(F){$\bullet$}
\drawvertex(G){$\bullet$}
\drawvertex(N){$\bullet$}\drawvertex(L){$\bullet$}

\end{picture}
\end{center}

In this case
$$
k(A) = k(A_1) +k(A_2) + k(A_3)\qquad \mbox{and} \qquad r(A) =
r(A_1) + r(A_2) + r(A_3).
$$
Moreover, one has in this case
\begin{eqnarray*}
n(A)
&=&(|E(A_1)|+|E(A_2)|+|E(A_3)|)-(|V(A_1)|+|V(A_2)|+|V(A_3)|)\\&+&
(k(A_1) +k(A_2) + k(A_3))\\ &=& n(A_1) + n(A_2) + n(A_3).
\end{eqnarray*}
Hence, for such a spanning subgraph $A$ of $\Sigma_{n+1}$ (of
\lq\lq fifth type\rq\rq), one gets:
$$
r(\Sigma_{n+1})-r(A) = \sum_{i=1}^3 (r(\Sigma_{n})-r(A_i))+2
\qquad \mbox { and } \qquad n(A)=n(A_1) + n(A_2) + n(A_3)
$$
and so
$$
(x-1)^{r(\Sigma_{n+1})-r(A)}(y-1)^{n(A)}= (x-1)^2\prod_{i=1}^3
(x-1)^{r(\Sigma_{n})-r(A_i)}(y-1)^{n(A_i)}.
$$

\begin{teo}\label{noname}
For each $n\geq 1$, the Tutte polynomial $H_n(x,y)$ of $\Sigma_n$
is given by
$$
H_n(x,y)=H_{2,n}(x,y)+3H_{1,n}(x,y)+H_{0,n}(x,y),
$$
where the polynomials $H_{2,n}(x,y)$, $H_{1,n}(x,y)$,
$H_{0,n}(x,y) \in \mathbb{Z}[x,y]$ satisfy the following recursive
relations:
\begin{eqnarray}\label{h2}
H_{2,n+1}(x,y)&=& (y-1)H_{2,n}^3 +
\frac{1}{x-1}\left(6H_{2,n}^2H_{1,n}+3H_{2,n}H_{1,n}^2\right) \\
&+&3H_{2,n}^3+6H_{2,n}^2H_{1,n}+3H_{2,n}H_{1,n}^2.\nonumber
\end{eqnarray}
\begin{eqnarray}\label{h1}
H_{1,n+1}(x,y)&=&
(y-1)H_{2,n}^2H_{1,n}\\&+&\frac{1}{x-1}\left(H_{2,n}^2H_{0,n}+7H_{2,n}H_{1,n}^2+2H_{2,n}H_{1,n}H_{0,n}+4H_{1,n}^3+H_{1,n}^2H_{0,n}\right)\nonumber\\
&+&7H_{2,n}^2H_{1,n}+2H_{2,n}^2H_{0,n}+14H_{2,n}H_{1,n}^2+4H_{2,n}H_{1,n}H_{0,n}\nonumber\\
&+&7H_{1,n}^3+2H_{1,n}^2H_{0,n}+
(x-1)\left(H_{2,n}^3+5H_{2,n}^2H_{1,n}+H_{2,n}^2H_{0,n}\right.\nonumber\\&+&
\left.7H_{2,n}H_{1,n}^2+2H_{2,n}H_{1,n}H_{0,n}+3H_{1,n}^3+H_{1,n}^2H_{0,n}\right)\nonumber
\end{eqnarray}
\begin{eqnarray}\label{h0}
H_{0,n+1}(x,y)&=& (y-1)\left(3H_{2,n}H_{1,n}^2+H_{1,n}^3\right)\\
&+&
\frac{1}{x-1}\left(12H_{2,n}H_{1,n}H_{0,n}+3H_{2,n}H_{0,n}^2+14H_{1,n}^3+24H_{1,n}^2H_{0,n}\right.\nonumber\\
&+&\left.9H_{1,n}H_{0,n}^2+H_{0,n}^3\right)+
3H_{2,n}^2H_{0,n}+36H_{2,n}H_{1,n}^2+42H_{2,n}H_{1,n}H_{0,n}\nonumber\\
&+&9H_{2,n}H_{0,n}^2+60H_{1,n}^3+75H_{1,n}^2H_{0,n}+
27H_{1,n}H_{0,n}^2 +3H_{0,n}^3\nonumber\\&+&
(x\!-\!1)\left(12H_{2,n}^2H_{1,n}+6H_{2,n}^2H_{0,n}+60H_{2,n}H_{1,n}^2+48H_{2,n}H_{1,n}H_{0,n}\right.\nonumber\\
&+&\left.9H_{2,n}H_{0,n}^2+72H_{1,n}^3+
78H_{1,n}^2H_{0,n}+27H_{1,n}H_{0,n}^2+3H_{0,n}^3\right)\nonumber\\
&+&
(x\!-\!1)^2\!\!\left(H_{2,n}^3\!+\!9H_{2,n}^2H_{1,n}\!+\!3H_{2,n}^2H_{0,n}\!+\!27H_{2,n}H_{1,n}^2\!+\!\!18H_{2,n}H_{1,n}H_{0,n}\right.\nonumber\\
&+&\left.3H_{2,n}H_{0,n}^2
+27H_{1,n}^3+27H_{1,n}^2H_{0,n}+9H_{1,n}H_{0,n}^2+H_{0,n}^3\right),\nonumber
\end{eqnarray}
with initial conditions
$$
H_{2,1}(x,y)=y+2 \qquad H_{1,1}(x,y)=x-1 \qquad H_{0,1}(x,y)=(x-1)^2.
$$
\end{teo}

\begin{proof}
The proof follows the same strategy as in Theorem
\ref{ricorsivesierpinski}. Observe that in this case we have
different powers of $(x-1)$ occurring, due to the different
possible number of special edges belonging to a spanning subgraph.
\end{proof}

For each $n\geq 1$, let us call $I_n$ the graph obtained by
$\Sigma_n$ by contracting only the special edges joining the three
copies of $\Sigma_{n-1}$, so that $I_n$ has the following
structure.

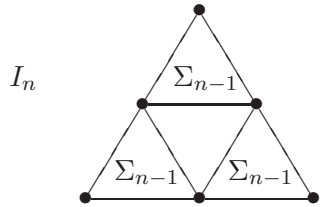
\begin{center}\unitlength=0.25mm
\begin{picture}(400,110)
\put(100,70){$I_n$}
\letvertex D=(200,110)\letvertex E=(170,60)\letvertex F=(140,10)\letvertex G=(200,10)
\letvertex H=(260,10)\letvertex I=(230,60)

\put(186,70){$\Sigma_{n-1}$}\put(155,20){$\Sigma_{n-1}$}\put(215,20){$\Sigma_{n-1}$}

\drawvertex(D){$\bullet$}
\drawvertex(E){$\bullet$}\drawvertex(F){$\bullet$}
\drawvertex(G){$\bullet$}\drawvertex(H){$\bullet$}
\drawvertex(I){$\bullet$}

\drawundirectededge(E,D){} \drawundirectededge(F,E){}
\drawundirectededge(G,F){} \drawundirectededge(H,G){}
\drawundirectededge(I,H){} \drawundirectededge(D,I){}
\drawundirectededge(I,E){} \drawundirectededge(E,G){}
\drawundirectededge(G,I){}
\end{picture}
\end{center}
In other words, the graph $I_n$ can be regarded as a
Sierpi\'{n}ski graph $\Gamma_n$, where each subgraph $G_1,G_2,G_3$
is isomorphic to $\Sigma_{n-1}$ and not to the graph
$\Gamma_{n-1}$.

The following proposition establishes a relationship between the
Tutte polynomial of $\Sigma_n$ and the Tutte polynomial of the
Sierpi\'{n}ski graph $\Gamma_n$, via the introduction of the Tutte
polynomial of $I_n$. More precisely, the following result holds.

\begin{prop}\label{sierp-hanoi}
For each $n\geq 1$, one has
$$
H_{n+1}(x,y) = (x^2+x+1)H_{n}^3(x,y) + T(I_{n+1};x,y).
$$
\end{prop}
\begin{proof}
We prove the assertion by using Property \eqref{prodotto} and the
deletion-contraction property of the Tutte polynomial (see
Definition \ref{contracting}). Let us start by choosing the bottom
special edge in $\Sigma_{n+1}$: then, by deletion and contraction,
we have
\begin{center}
\begin{picture}(400,80)
\letvertex A=(200,75)
\letvertex B=(185,50) \letvertex F=(215,50)
\letvertex C=(170,25)  \letvertex E=(230,25)
\letvertex D=(155,0)  \letvertex G=(245,0)
\letvertex N=(185,0)  \letvertex L=(215,0)

\put(54,40){$H_{n+1}(x,y)\ \ = \ \ T\, ( $} \put(238,40){$)=$}

\put(195,55){$\Sigma_n$} \put(165,5){$\Sigma_n$}
\put(225,5){$\Sigma_n$}

\drawundirectededge(D,G){} \drawundirectededge(A,D){}
\drawundirectededge(A,G){}
   \drawundirectededge(L,E){}
\drawundirectededge(B,F){} \drawundirectededge(C,N){}

\drawvertex(A){$\bullet$}\drawvertex(B){$\bullet$}
\drawvertex(C){$\bullet$}
 \drawvertex(D){$\bullet$}
\drawvertex(E){$\bullet$}\drawvertex(F){$\bullet$}
\drawvertex(G){$\bullet$}
\drawvertex(N){$\bullet$}\drawvertex(L){$\bullet$}

\end{picture}
\end{center}

%%%%%%%%%%%%%%%%%%%%%%%%%%%%%%%%%%%%%%%%%%%%%%%%%%%%%%%%%%%%%%%%%%%%%%%%%%%%%%%
\begin{center}
\begin{picture}(400,80)
\letvertex A=(100,75)
\letvertex B=(85,50) \letvertex F=(115,50)
\letvertex C=(70,25)  \letvertex E=(130,25)
\letvertex D=(55,0)  \letvertex G=(145,0)
\letvertex N=(85,0)  \letvertex L=(115,0)

\letvertex a=(300,75)
\letvertex b=(285,50) \letvertex f=(315,50)
\letvertex c=(270,25)  \letvertex e=(330,25)
\letvertex d=(255,0)  \letvertex g=(345,0)
\letvertex n=(300,0)  \letvertex l=(300,0)

\put(95,55){$\Sigma_n$} \put(125,5){$\Sigma_n$}
\put(65,5){$\Sigma_n$}

\put(295,55){$\Sigma_n$} \put(268,5){$\Sigma_n$}
\put(323,5){$\Sigma_n$}

\put(45,40){$T\, ( $} \put(138,40){$)$} \put(195,40){$+ $}

\put(245,40){$T\, ( $} \put(335,40){$).$}

\drawundirectededge(D,N){} \drawundirectededge(L,G){}
\drawundirectededge(A,D){} \drawundirectededge(A,G){}
   \drawundirectededge(L,E){}
\drawundirectededge(B,F){} \drawundirectededge(C,N){}

\drawundirectededge(d,g){} \drawundirectededge(a,d){}
\drawundirectededge(a,g){}
   \drawundirectededge(l,e){}
\drawundirectededge(b,f){} \drawundirectededge(c,n){}

\drawvertex(A){$\bullet$}\drawvertex(B){$\bullet$}
\drawvertex(C){$\bullet$}
 \drawvertex(D){$\bullet$}
\drawvertex(E){$\bullet$}\drawvertex(F){$\bullet$}
\drawvertex(G){$\bullet$}
\drawvertex(N){$\bullet$}\drawvertex(L){$\bullet$}

\drawvertex(a){$\bullet$}\drawvertex(b){$\bullet$}
\drawvertex(c){$\bullet$}
 \drawvertex(d){$\bullet$}
\drawvertex(e){$\bullet$}\drawvertex(f){$\bullet$}
\drawvertex(g){$\bullet$}
\drawvertex(n){$\bullet$}\drawvertex(l){$\bullet$}
\end{picture}
\end{center}

Next, in order to compute the Tutte polynomial, we can use
Property \eqref{prodotto} for the graph on the left and, for the
graph on the right, we can apply again the deletion-contraction
argument, with respect to the left special edge. Thus, we get:

\begin{center}
\begin{picture}(400,80)
\letvertex A=(70,65)
\letvertex B=(55,40) \letvertex F=(85,40)

\letvertex a=(240,75)
\letvertex b=(225,50) \letvertex f=(255,50)
\letvertex c=(210,25)  \letvertex e=(270,25)
\letvertex d=(195,0)  \letvertex g=(285,0)
\letvertex n=(240,0)  \letvertex l=(240,0)

\letvertex p=(370,75)
\letvertex q=(355,50) \letvertex r=(385,50)
\letvertex s=(340,25)  \letvertex t=(400,25)
\letvertex u=(325,0)  \letvertex v=(370,0)
\letvertex w=(370,0)  \letvertex z=(415,0)

\put(-60,40){$H_{n+1}(x,y)=x^2T\, ( $} \put(93,40){$)^3$}
\put(135,40){$+ $}

\put(65,45){$\Sigma_n$}

\put(235,55){$\Sigma_n$} \put(208,5){$\Sigma_n$}
\put(263,5){$\Sigma_n$}

\put(365,55){$\Sigma_n$} \put(338,5){$\Sigma_n$}
\put(393,5){$\Sigma_n$}

\put(185,40){$T\, ( $} \put(275,40){$)$}\put(290,40){$+$}

\put(310,40){$T\, ( $} \put(410,40){$)=$}

\drawundirectededge(A,B){} \drawundirectededge(F,B){}
\drawundirectededge(A,F){}

\drawundirectededge(d,g){}
\drawundirectededge(a,b){}\drawundirectededge(c,d){}
 \drawundirectededge(a,g){}\drawundirectededge(e,g){}
   \drawundirectededge(l,e){}
\drawundirectededge(b,f){} \drawundirectededge(c,n){}

\drawundirectededge(p,u){}
\drawundirectededge(u,z){}\drawundirectededge(p,z){}
 \drawundirectededge(r,s){}
   \drawundirectededge(s,v){}
\drawundirectededge(t,w){}

\drawvertex(A){$\bullet$}\drawvertex(B){$\bullet$}
\drawvertex(F){$\bullet$}

\drawvertex(a){$\bullet$} \drawvertex(b){$\bullet$}
\drawvertex(c){$\bullet$}
 \drawvertex(d){$\bullet$}
\drawvertex(e){$\bullet$}\drawvertex(f){$\bullet$}
\drawvertex(g){$\bullet$}
\drawvertex(n){$\bullet$}\drawvertex(l){$\bullet$}

\drawvertex(p){$\bullet$}\drawvertex(v){$\bullet$}

 \drawvertex(r){$\bullet$}
\drawvertex(s){$\bullet$}\drawvertex(w){$\bullet$}
\drawvertex(t){$\bullet$}
\drawvertex(u){$\bullet$}\drawvertex(z){$\bullet$}
\end{picture}
\end{center}

\begin{center}
\begin{picture}(400,80)
\letvertex A=(90,65)
\letvertex B=(75,40) \letvertex F=(105,40)

\letvertex p=(240,75)
\letvertex q=(225,50) \letvertex r=(255,50)
\letvertex s=(210,25)  \letvertex t=(270,25)
\letvertex u=(195,0)  \letvertex v=(240,0)
\letvertex w=(240,0)  \letvertex z=(285,0)

\letvertex a=(370,75)
%\letvertex b=(355,50) \letvertex f=(385,50)
\letvertex c=(340,25)  \letvertex e=(400,25)
\letvertex d=(325,0)  \letvertex g=(370,0)
\letvertex n=(370,0)  \letvertex l=(415,0)

\put(85,45){$\Sigma_n$}

\put(235,51){$\Sigma_n$} \put(208,5){$\Sigma_n$}
\put(261,5){$\Sigma_n$}

\put(365,51){$\Sigma_n$} \put(338,5){$\Sigma_n$}
\put(393,5){$\Sigma_n$}

\put(,40){$(x^2 +x)T\, ( $}  \put(140,40){$+ $}

\put(180,40){$T\, ( $} \put(114,40){$)^3$}

\put(310,40){$T\, ( $} \put(410,40){$)$}

\drawundirectededge(A,B){} \drawundirectededge(F,B){}
\drawundirectededge(A,F){}

\drawundirectededge(d,n){}
\drawundirectededge(a,c){}\drawundirectededge(c,d){}
 \drawundirectededge(a,e){}
\drawundirectededge(e,g){}
   \drawundirectededge(l,g){}
\drawundirectededge(c,e){} \drawundirectededge(c,n){}
\drawundirectededge(e,l){}

\drawundirectededge(p,u){}
\drawundirectededge(u,z){}\drawundirectededge(p,r){}
\drawundirectededge(t,z){}
 \drawundirectededge(r,s){}
   \drawundirectededge(s,v){}
\drawundirectededge(t,w){}

\drawvertex(A){$\bullet$}\drawvertex(B){$\bullet$}
\drawvertex(F){$\bullet$}

\drawvertex(a){$\bullet$}
%\drawvertex(b){$\bullet$}
\drawvertex(c){$\bullet$}
 \drawvertex(d){$\bullet$}
\drawvertex(e){$\bullet$}
%\drawvertex(f){$\bullet$}
\drawvertex(g){$\bullet$}
\drawvertex(n){$\bullet$}\drawvertex(l){$\bullet$}
\put(273,40){$)$}\put(290,40){$+$}
\drawvertex(p){$\bullet$}\drawvertex(v){$\bullet$}

 \drawvertex(r){$\bullet$}
\drawvertex(s){$\bullet$}\drawvertex(w){$\bullet$}
\drawvertex(t){$\bullet$}
\drawvertex(u){$\bullet$}\drawvertex(z){$\bullet$}
\end{picture}
\end{center}
where the last equality is obtained by using Property
\eqref{prodotto} and then by applying the deletion-contraction
argument with respect to the right special edge. Finally, we can
apply again Property \eqref{prodotto} for the middle graph, and
then we can observe that the graph $I_{n+1}$ appeared on the
right, so that we get:
\begin{center}
\begin{picture}(300,60)

\letvertex A=(110,40)
\letvertex B=(95,15) \letvertex F=(125,15)

\letvertex a=(280,55)\letvertex b=(265,30)
\letvertex c=(250,5)\letvertex d=(280,5)
\letvertex e=(310,5)\letvertex f=(295,30)

\put(105,20){$\Sigma_n$}

\put(275,35){$\Sigma_n$} \put(260,10){$\Sigma_n$}
\put(290,10){$\Sigma_n$}

\put(-73,15){$H_{n+1}(x,y)=(x^2 +x +1)T\, ( $} \put(130,15){$)^3$}
\put(171,15){$+ $}

\put(210,15){$T\, ( $} \put(330,15){$)$}

\drawundirectededge(A,B){} \drawundirectededge(F,B){}
\drawundirectededge(A,F){}

\drawundirectededge(a,b){} \drawundirectededge(b,c){}
\drawundirectededge(c,d){} \drawundirectededge(d,e){}
\drawundirectededge(e,f){}\drawundirectededge(f,a){}\drawundirectededge(b,f){}\drawundirectededge(b,d){}\drawundirectededge(d,f){}

\drawvertex(a){$\bullet$}\drawvertex(b){$\bullet$}
\drawvertex(c){$\bullet$}\drawvertex(d){$\bullet$}
\drawvertex(e){$\bullet$}\drawvertex(f){$\bullet$}
\drawvertex(A){$\bullet$}\drawvertex(B){$\bullet$}
\drawvertex(F){$\bullet$}
\end{picture}
\end{center}
and so $H_{n+1}(x,y) = (x^2+x+1)H_n^3(x,y) + T(I_{n+1};x,y)$, as
required.
\end{proof}
\smallskip
\begin{os}\rm
Since the graph $I_{n+1}$ can be regarded as the Sierpi\'{n}ski
graph $\Gamma_{n+1}$, where each subgraph $G_1,G_2,G_3$ is
isomorphic to $\Sigma_{n}$, it is clear that $T(I_{n+1};x,y)$ is
given by $T_{n+1}(x,y)$, obtained from Equations \eqref{t2n},
\eqref{t1n} and \eqref{t0n}, where $T_{2,n}, T_{1,n}$ and
$T_{0,n}$ have to be replaced by $H_{2,n}$, $H_{1,n}$ and
$H_{0,n}$, respectively. Moreover, the terms of $T(I_{n+1};x,y)$
in $H_{n+1}(x,y)$ are exactly the terms of $H_{n+1}(x,y)$ having a
factor $(y-1)$ or $\frac{1}{x-1}$, i.e., the terms of first and
second type corresponding to subgraphs of $\Sigma_{n+1}$
containing the three special edges.
\end{os}

The following lemma can be easily proven by induction, using
Equations \eqref{h1} and \eqref{h0}.

\begin{lem}\label{lemmafattorihanoi}
For each $n\geq 1$, $x-1$ divides $H_{1,n}(x,y)$ and $(x-1)^2$
divides $H_{0,n}(x,y)$ in $\mathbb{Z}[x,y]$.
\end{lem}
As a consequence, we can write
\begin{eqnarray}\label{semplificatehanoi}
H_{1,n}(x,y) = (x-1)N_n(x,y) \qquad \mbox{ and } \qquad
H_{0,n}(x,y) = (x-1)^2M_n(x,y),
\end{eqnarray}
with $N_n(x,y)$ and $ M_n(x,y) \in \mathbb{Z}[x,y]$.

\indent Using \eqref{semplificatehanoi} for $H_{1,n}(x,y)$ and
$H_{0,n}(x,y)$, Equations \eqref{h2}, \eqref{h1} and \eqref{h0}
can be rewritten as

\begin{eqnarray}\label{pigroni2}
H_{2,n+1}(x,y) &=& (y-1)H_{2,n}^3+3H_{2,n}^3+6H_{2,n}^2N_{n}\\
&+& (x-1)\left(6H_{2,n}^2N_{n}+3H_{2,n}N_{n}^2\right) +
3(x-1)^2H_{2,n}N_{n}^2\nonumber
\end{eqnarray}

\begin{eqnarray}\label{pigroni1}
N_{n+1}(x,y) &=&
(y-1)H_{2,n}^2N_{n}+H_{2,n}^3+7H_{2,n}^2N_{n}+H_{2,n}^2M_{n}+7H_{2,n}N_{n}^2\\
&+&(x-1)\left(5H_{2,n}^2N_{n}+2H_{2,n}^2M_{n}+14H_{2,n}N_{n}^2+2H_{2,n}N_{n}M_{n}+4N_{n}^3\right)\nonumber\\
&+&(x-1)^2\left(H_{2,n}^2M_{n}+7H_{2,n}N_{n}^2+4H_{2,n}N_{n}M_{n}+7N_{n}^3+N_{n}^2M_{n}\right)\nonumber\\
&+&(x-1)^3
\left(2H_{2,n}N_{n}M_{n}+3N_{n}^3+2N_{n}^2M_{n}\right)\nonumber\\
&+& (x-1)^4N_{n}^2M_{n}\nonumber
\end{eqnarray}

\begin{eqnarray}\label{pigroni0}
M_{n+1}(x,y) \!\!&=&
3(y-1)H_{2,n}N_{n}^2+H_{2,n}^3+12H_{2,n}^2N_{n}+3H_{2,n}^2M_{n}+36H_{2,n}N_{n}^2\\
&+&12H_{2,n}N_{n}M_{n}+14N_{n}^3\nonumber\\
&+&(x-1)\!\left((y-1)N_{n}^3\!+\!9H_{2,n}^2N_{n}\!+\!6H_{2,n}^2M_{n}\!+\!60H_{2,n}N_{n}^2\!+\!42H_{2,n}N_{n}M_{n}\right.\nonumber\\
&+&\left. 3H_{2,n}M_{n}^2+60N_{n}^3+24N_{n}^2M_{n}
\right)\nonumber\\
&+&
(x-1)^2\left(3H_{2,n}^2M_{n}+27H_{2,n}N^2_{n}+48H_{2,n}N_{n}M_{n}+9
H_{2,n}M_{n}^2\right.\nonumber\\
&+&\left.72N_{n}^3+75N_{n}^2M_{n}+9N_{n}M_{n}^2\right)\nonumber\\
&+& (x-1)^3\!\left(18H_{2,n}N_{n}M_{n}\!+\!9H_{2,n}M_{n}^2\!+\!27N_{n}^3\!+\!78N_{n}^2M_n\!+\!27N_{n}M_{n}^2\!+\!M_{n}^3\right)\nonumber\\
&+& (x-1)^4
\left(3H_{2,n}M_{n}^2+27N_{n}^2M_{n}+27N_{n}M_{n}^2+3M_{n}^3\right)\nonumber\\
&+&
(x-1)^5\left(9N_{n}M_{n}^2+3M_{n}^3\right)+(x-1)^6M_{n}^3,\nonumber
\end{eqnarray}
with initial conditions
$$
H_{2,1}(x,y)=y+2 \qquad N_1(x,y)= M_1(x,y) =1.
$$
As in Section \ref{sezione serpin}, we will use these reduced
formulas to compute several evaluations of the Tutte polynomial.
Let us start by writing the reliability polynomial
$R(\Sigma_n,p)$.

\begin{prop}\label{donatellahanoi}
For each $n\geq 1$, the reliability polynomial $R(\Sigma_n,p)$ is
given by
$$
R(\Sigma_n,p) =
p^{3^n-1}(1-p)^{\frac{3^n-1}{2}}H_n\left(1,\frac{1}{1-p}\right),
$$
with $H_n\left(1,\frac{1}{1-p}\right)=
H_{2,n}\left(1,\frac{1}{1-p}\right)$ and

\begin{eqnarray}\label{h2ridottein1}
H_{2,n+1}\left(1,\frac{1}{1-p}\right) &=&
\frac{p}{1-p}H_{2,n}^3+3H_{2,n}^3+6H_{2,n}^2N_{n}
\end{eqnarray}

\begin{eqnarray}\label{h1ridottein1}
N_{n+1}\left(1,\frac{1}{1-p}\right) &=&
\frac{p}{1-p}H_{2,n}^2N_{n}+H_{2,n}^3+7H_{2,n}^2N_{n}+H_{2,n}^2M_{n}+7H_{2,n}N_{n}^2
\end{eqnarray}

\begin{eqnarray}\label{h0ridottein1}
M_{n+1}\left(1,\frac{1}{1-p}\right) &=&
\frac{3p}{1-p}H_{2,n}N_{n}^2+H_{2,n}^3+12H_{2,n}^2N_{n}+3H_{2,n}^2M_{n}\\&+&36H_{2,n}N_{n}^2
+12H_{2,n}N_{n}M_{n}+14N_{n}^3,\nonumber
\end{eqnarray}
with initial conditions
$$
H_{2,1}\left(1,\frac{1}{1-p}\right)=\frac{3-2p}{1-p} \qquad
N_1\left(1,\frac{1}{1-p}\right)= M_1\left(1,\frac{1}{1-p}\right) =
1.
$$
\end{prop}

\begin{proof}
By Lemma \ref{lemmafattorihanoi}, one has
$H_{1,n}(1,y)=H_{0,n}(1,y)=0$, for every $y\in \mathbb{R}$.
Therefore, $H_n\left(1,\frac{1}{1-p}\right)=
H_{2,n}\left(1,\frac{1}{1-p}\right)$; then it is enough to apply
(1) of Theorem \ref{twopolynomials} and use Equations
\eqref{pigroni2}, \eqref{pigroni1} and \eqref{pigroni0}.
\end{proof}

\begin{prop}\label{propcomplexhanoi}
The complexity $\tau(\Sigma_n)$ is $H_n(1,1)= H_{2,n}(1,1)$, where
\begin{eqnarray*}
H_{2,n+1}(1,1) &=&3H_{2,n}^3+ 6H_{2,n}^2N_{n}
\end{eqnarray*}
\begin{eqnarray*}
N_{n+1}(1,1)
&=&H_{2,n}^3+7H_{2,n}^2N_{n}+H_{2,n}^2M_{n}+7H_{2,n}N_{n}^2
\end{eqnarray*}
\begin{eqnarray*}
M_{n+1}(1,1)
=H_{2,n}^3+12H_{2,n}^2N_{n}+3H_{2,n}^2M_{n}+36H_{2,n}N_{n}^2+
12H_{2,n}N_{n}M_{n}+14N_{n}^3,
\end{eqnarray*}
with initial conditions
$$
H_{2,1}(1,1)=3 \qquad N_1(1,1)= M_1(1,1) =1.
$$
\end{prop}

\begin{proof}
The complexity of $\Sigma_n$ is obtained by evaluating
$H_n\left(1,\frac{1}{1-p}\right)=H_{2,n}\left(1,\frac{1}{1-p}\right)$
in $p=0$, using Equations \eqref{h2ridottein1},
\eqref{h1ridottein1} and \eqref{h0ridottein1}.
\end{proof}

\begin{os}\rm
These formulas coincide with the relations obtained in
\cite{noispanning}, without using Tutte polynomials. More
precisely, one can find in \cite[Proposition 3.4]{noispanning}:
\begin{enumerate}
\item $ H_n(1,1)=\tau(\Sigma_n)= 3^{\frac{3^n+2n-1}{4}}5^{\frac{3^{n}-2n-1}{4}}$;
\item $N_n(1,1)= 3^{\frac{3^{n}-2n-1}{4}}5^{\frac{3^{n}-2n-1}{4}} \cdot \frac{5^n-3^n}{2}$;
\item $M_n(1,1) = 3^{\frac{3^{n}-6n+3}{4}}5^{\frac{3^{n}-2n-1}{4}}\cdot\left(\frac{5^n-3^n}{2}\right)^2$.
\end{enumerate}
Then, the asymptotic growth constant of the spanning trees of
$\Sigma_n$ is
$$
\lim_{n\to \infty}\frac{\log(\tau(\Sigma_n))}{|V(\Sigma_n)|}=
\frac{1}{4}\left(\log 3+\log 5\right).
$$
\end{os}
\vspace{0.5cm}

Evaluating $H_n\left(1,\frac{1}{1-p}\right)$ in $p=\frac{1}{2}$
gives the number of connected spanning subgraphs of $\Sigma_n$.
\begin{prop}\label{propconnsubgrhanoi}
The number of connected spanning subgraphs of $\Sigma_n$ is given
by $H_n(1,2) = H_{2,n}(1,2)$, with
\begin{eqnarray*}
H_{2,n+1}(1,2) &=& 4H_{2,n}^3+6H_{2,n}^2N_{n}
\end{eqnarray*}
\begin{eqnarray*}
N_{n+1}(1,2) &=&H_{2,n}^3+
8H_{2,n}^2N_{n}+H_{2,n}^2M_{n}+7H_{2,n}N_{n}^2
\end{eqnarray*}
\begin{eqnarray*}
M_{n+1}(1,2) =H_{2,n}^3+12H_{2,n}^2N_{n}+3H_{2,n}^2M_{n}+
39H_{2,n}N_{n}^2 +12H_{2,n}N_{n}M_{n}+14N_{n}^3,
\end{eqnarray*}
with initial conditions
$$
H_{2,1}(1,2)=4 \qquad N_1(1,2)= M_1(1,2)=1.
$$
\end{prop}

\begin{proof}
One has $H_n(1,2)=H_{2,n}(1,2)$ since $H_{1,n}(1,y) =
H_{0,n}(1,y)=0$, for every $y\in \mathbb{R}$ (see Lemma
\ref{lemmafattorihanoi}). Then it suffices to use Formula (2) of
Theorem \ref{evaluations}.
\end{proof}

The following proposition about the number of spanning forests of
$\Sigma_n$ holds.

\begin{prop}\label{spannforesthanoi}
The number of spanning forests of $\Sigma_n$ is given by
$$
H_{n}(2,1) = H_{2,n}(2,1) + 3N_n(2,1)+M_n(2,1),
$$
where
\begin{eqnarray*}
H_{2,n+1}(2,1)= 3H_{2,n}^3+12H_{2,n}^2N_{n}+6H_{2,n}N_{n}^2
\end{eqnarray*}
\begin{eqnarray*}
N_{n+1}(2,1)&=&H_{2,n}^3+12H_{2,n}^2N_n+
4H_{2,n}^2M_{n}+28H_{2,n}N_{n}^2+8H_{2,n}N_{n}M_{n}+14N_{n}^3
+4N_{n}^2M_{n}
\end{eqnarray*}
\begin{eqnarray*}
M_{n+1}(2,1)&=&H_{2,n}^3+21H_{2,n}^2N_n+12H_{2,n}^2M_n+123H_{2,n}N_n^2+120H_{2,n}N_{n}M_{n}+24H_{2,n}M_{n}^2\\&+&173N_{n}^3
+204N_{n}^2M_{n}+72N_{n}M_{n}^2+8M_{n}^3,
\end{eqnarray*}
with initial conditions
$$
H_{2,1}(2,1)=3 \qquad N_1(2,1)=M_1(2,1)=1.
$$
\end{prop}

\begin{proof}
It suffices to apply Formula (3) of Theorem \ref{evaluations} and
to observe that $H_{1,n}(2,y)= N_n(2,y)$ and
$H_{0,n}(2,y)=M_n(2,y)$, for each $y\in \mathbb{R}$.
\end{proof}

Similarly to the case of Sierpi\'{n}ski graphs, the following
proposition holds.
\begin{prop}
$H_n(2,2) = 2^{ |E(\Sigma_{n})|} = 2^{\frac{3^{n+1}-3}{2}}$.
\end{prop}

\begin{proof}
The proof is by induction. For $n=1$, we have $H_1(2,2)=8=2^3=2^{
|E(\Sigma_{1})|}$.  Then, we recall that
$|E(\Sigma_{n+1})|=3|E(\Sigma_{n})| +3$. An easy computation shows
that $H_{n+1}(2,2)= 8H_{n}(2,2)^3$; therefore, $H_{n+1}(2,2)=
8H_{n}(2,2)^3= \left(2\cdot 2^{ |E(\Sigma_{n})|}\right)^3=
 2^{ 3+ 3|E(\Sigma_{n})|}= 2^{ |E(\Sigma_{n+1})|} $.
\end{proof}

As regards the number of acyclic orientations of $\Sigma_n$, we
have the following result.

\begin{prop}\label{propacyclhanoi}
The number of acyclic orientations on $\Sigma_{n}$ is given by
$H_n(2,0)$, where
$$
H_{n+1}(2,0) = (2H_{n}(2,0))^3 -
2\left(H_{2,n}(2,0)+N_{n}(2,0)\right)^3
$$
and
\begin{eqnarray*}
H_{2,n+1}(2,0)=2H_{2,n}^3+12H_{2,n}^2N_{n}+6H_{2,n}N_{n}^2
\end{eqnarray*}
\begin{eqnarray*}
N_{n+1}(2,0)&=&H_{2,n}^3+11H_{2,n}^2N_n+
4H_{2,n}^2M_{n}+28H_{2,n}N_{n}^2+8H_{2,n}N_{n}M_{n}+14N_{n}^3
+4N_{n}^2M_{n}
\end{eqnarray*}
\begin{eqnarray*}
M_{n+1}(2,0)&=&H_{2,n}^3+21H_{2,n}^2N_n+12H_{2,n}^2M_n+120H_{2,n}N_n^2+120H_{2,n}N_{n}M_{n}+24H_{2,n}M_{n}^2\\&+&172N_{n}^3+
204N_{n}^2M_{n}+72N_{n}M_{n}^2+8M_{n}^3,
\end{eqnarray*}
with initial conditions
$$
H_{2,1}(2,0)=2 \qquad N_1(2,0)=M_1(2,0)=1.
$$
\end{prop}

\begin{proof}
It suffices to apply Formula (5) of Theorem \ref{evaluations}.
Then, one can directly verify that the Tutte polynomial
$H_{2,n+1}(2,0) + 3H_{1,n+1}(2,0) +H_{0,n+1}(2,0)$ can be
rewritten as $(2H_{n}(2,0))^3 -
2\left(H_{2,n}(2,0)+N_{n}(2,0)\right)^3$.
\end{proof}

Next, let us look at the chromatic polynomial of $\Sigma_n$.
\begin{prop}\label{propchromhanoi}
For each $n\geq 1$, the chromatic polynomial $\chi_n(\lambda)$ of
the Schreier graph $\Sigma_n$ is
$$
\chi_n(\lambda) = (-1)^{3^n-1}\lambda P_n(\lambda),
$$
where $P_n(\lambda) = P_{2,n}(\lambda) + 3P_{1,n}(\lambda) +
P_{0,n}(\lambda)$, with $P_{i,n}(\lambda)=H_{i,n}(1-\lambda,0)$,
for each $i=1,2,3$.
\end{prop}

\begin{proof}
It is an easy consequence of Equation (2) of Theorem
\ref{twopolynomials}. Here we omit the explicit recursive
equations of $P_{i,n}(\lambda)$.
\end{proof}

Also in the case of the Schreier graph $\Sigma_n$, we can
explicitly study the relationship between the evaluation of the
Tutte polynomial on the hyperbola $(x-1)(y-1)=2$ and the partition
function of the Ising model. In \cite[Theorem 3.3]{noiising}, the
partition function of the Ising model on $\Sigma_n$ has been
described as
$$
Z_n=2^{3^n}\cosh(\beta J)^\frac{3^{n+1}-3}{2}\Psi_n(\tanh(\beta
J)),
$$
with
$$
\Psi_n(z)=z^{3^n}\prod_{k=1}^n\psi_k^{3^{n-k}}(z)(\psi_{n+1}(z)-1),
$$
where $\psi_1(z) = \frac{z+1}{z}$ and $\psi_k(z) =
\psi_{k-1}^2(z)-3\psi_{k-1}(z)+4$, for each $k\geq 2$.

\begin{teo}\label{thmisinghanoistrano}
For each $n\geq 1$, one has
\begin{eqnarray}\label{isingbasehanoi}
2(e^{2\beta J}-1)^{|V(\Sigma_n)|-1}e^{-\beta
J|E(\Sigma_n)|}H_n\left(\frac{e^{2\beta J}+1}{e^{2\beta
J}-1},e^{2\beta J}\right) = Z_n.
\end{eqnarray}
\end{teo}

\begin{proof}
The proof can be done by induction and follows the same strategy
as for the Sierpi\'{n}ski graphs (see Theorem
\ref{thmisingsierp}). Recall that $|E(\Sigma_n)| =
\frac{3^{n+1}-3}{2}$ and $|V(\Sigma_n)|= 3^n$. Putting, as usual,
$e^{\beta J}=t$, Equation \eqref{isingbasehanoi} can be written as
$$
\frac{2(t^2-1)^{3^n-1}}{t^{\frac{3^{n+1}-3}{2}}}H_n\left(\frac{t^2+1}{t^2-1},t^2\right)=
2^{3^n}\left(\frac{t^2+1}{2t}\right)^{\frac{3^{n+1}-3}{2}}\Psi_n\left(\frac{t^2-1}{t^2+1}\right),
$$
or, more explicitly,
$$
\frac{2(t^2\!-\!1)^{3^n-1}}{t^{\frac{3^{n+1}-3}{2}}}H_n\!\!\left(\frac{t^2\!+\!1}{t^2\!-\!1},t^2\right)\!=\!2^{3^n}\!\!
\left(\frac{t^2\!+\!1}{2t}\right)^{\frac{3^{n+1}-3}{2}}
\!\!\left(\frac{t^2\!-\!1}{t^2\!+\!1}\right)^{3^n}\!\!
\prod_{k=1}^n\psi_k^{3^{n-k}}(z)(\psi_{n+1}(z)-1)\left|_{z=\frac{t^2-1}{t^2+1}}\right..
$$
For each $n\geq 1$, set
$$
A_n'(x,y)= H_{2,n}(x,y)+H_{1,n}(x,y) \qquad
B_n'(x,y)=2H_{1,n}(x,y)+H_{0,n}(x,y)
$$
and
$$
C_n' =  \frac{2(t^2-1)^{3^n-1}}{t^{\frac{3^{n+1}-3}{2}}}\qquad
\qquad
D_n'=2^{3^n}\left(\frac{t^2+1}{2t}\right)^{\frac{3^{n+1}-3}{2}}
\left(\frac{t^2-1}{t^2+1}\right)^{3^n-1}\!\!\!\!\!\!\!\!\!.
$$
Since $H_n(x,y)=A_n'(x,y)+B_n'(x,y)$, it is enough to prove the
equations
\begin{eqnarray*}
C_n'A_n'\left(\frac{t^2+1}{t^2-1}, t^2\right)
=D_n'\prod_{k=1}^n\psi_k^{3^{n-k}}(z)\left|_{z=\frac{t^2-1}{t^2+1}}\right.
\end{eqnarray*}
and
\begin{eqnarray*}
C_n'B_n'\left(\frac{t^2+1}{t^2-1}, t^2\right) = D_n'
\prod_{k=1}^n\psi_k^{3^{n-k}}(z)(z(\psi_{n+1}(z)-1)-1)\left|_{z=\frac{t^2-1}{t^2+1}}\right..
\end{eqnarray*}
They can be proven by induction, using the following relations
obtained by the evaluations of Equations \eqref{h2}, \eqref{h1}
and \eqref{h0} on the hyperbola $(x-1)(y-1)=2$:
$$
A'_{n+1}= \frac{(y+1)A_n^{'2}(B'_n+yB'_n+2yA'_n)}{2(y-1)}
$$
$$
B'_{n+1}=
\frac{(B'_n+yB'_n+2yA'_n)(4yA'_nB'_n+y^2A'_nB'_n+y^2B_n^{'2}+3A'_nB'_n+2yB_n^{'2}+4A_n^{'2}+B_n^{'2})}{2(y-1)^2},
$$
with initial conditions
$$
A'_1\left(\frac{y+1}{y-1},y\right)=\frac{y(y+1)}{y-1} \qquad
B'_1\left(\frac{y+1}{y-1},y\right)= \frac{4y}{(y-1)^2}.
$$
\end{proof}

\begin{center}
\textbf{Acknowledgements}
\end{center}
We wish to express our deepest gratitude to Tullio
Ceccherini-Silberstein and Toni Machì for useful comments and
suggestions. We are very grateful to the referee for several
remarks which improved the presentation of the paper.

\end{document}